\documentclass[10pt,reqno]{amsart}
\usepackage{url}
\usepackage{amsthm}
\usepackage{fullpage}
\usepackage{amsmath}
\usepackage{amssymb}
\usepackage{amsfonts}

\usepackage{enumitem}
\usepackage{mathscinet}
\usepackage{lipsum}
\usepackage[english]{babel}
\usepackage[autostyle]{csquotes}
\usepackage{bbold}
\usepackage{hyperref}

\hypersetup{
    colorlinks=true,
    linkcolor=red,
    citecolor=blue,
    }

\usepackage{graphicx}

\usepackage[utf8]{inputenc}
\usepackage[english]{babel}

\usepackage{color}
\usepackage{comment}

\newtheorem{thm}{Theorem}[section]
\newtheorem{definition}[thm]{Definition}
\newtheorem{theorem}[thm]{Theorem}
\newtheorem*{oseledec*}{Oseledec Theorem}
\newtheorem{cor}[thm]{Corollary}

\newtheorem{lemma}[thm]{Lemma}
\newtheorem{prop}[thm]{Proposition}

\makeatletter
\def\moverlay{\mathpalette\mov@rlay}
\def\mov@rlay#1#2{\leavevmode\vtop{%
   \baselineskip\z@skip \lineskiplimit-\maxdimen
   \ialign{\hfil$\m@th#1##$\hfil\cr#2\crcr}}}
\newcommand{\charfusion}[3][\mathord]{
    #1{\ifx#1\mathop\vphantom{#2}\fi
        \mathpalette\mov@rlay{#2\cr#3}
      }
    \ifx#1\mathop\expandafter\displaylimits\fi}
\makeatother

\newcommand{\bigcupdot}{\charfusion[\mathop]{\bigcup}{\cdot}}

\newcommand{\nocontentsline}[3]{}
\newcommand{\tocless}[2]{\bgroup\let\addcontentsline=\nocontentsline#1{#2}\egroup}

\usepackage{wasysym}

\def\Jac{\ensuremath{\mathrm{Jac}}}

\def\Vol{\ensuremath{\mathrm{Vol}}}
\def\E{\ensuremath{\mathcal{E}}}

\def\Vcs{\ensuremath{V^\mathrm{cs}}}

\title{Exponentially Fast Volume Limits} 
\author{S. Ben Ovadia$^{\star}$, F. Rodriguez-Hertz$^{\dagger}$}

\thanks{$^{\star}$ Department of Mathematics, Eberly College of Science, Pennsylvania State University, snir.benovadia@psu.edu \\ $\text{ }$$\text{ }$ $\text{ }$ $^\dagger$ Department of Mathematics, Eberly College of Science, Pennsylvania State University, fjr11@psu.edu}

\begin{document}
\maketitle
\begin{abstract}
Let $M$ be a $d$-dimensional closed Riemannian manifold, let $f\in\mathrm{Diff}^{1+\beta}(M)$, and denote by $m$ the Riemannian volume form of $M$. We prove that if $m\circ f^{-n}\xrightarrow[n\to\infty]{}\mu$ exponentially fast (see \textsection \ref{mainResults}), then $\mu$ is an SRB measure.

\end{abstract}

\section{Introduction}
\subsection{Motivation}
An important object in smooth ergodic theory is SRB measures, named after Sinai, Ruelle, and Bowen. SRB measures are invariant measures whose conditional measures on unstable leaves are absolutely continuous w.r.t the induced Riemannian volume on unstable leaves (see \cite{YoungSRBsurvey} for more details and properties of SRB measures).

Aside for potential physicality and compatibility with the Riemannian volume in dissipative systems, SRB measures are important as possible limit points of the Riemannian volume under the dynamics. In \cite{B4} Bowen shows that for Axiom A attractors which support an SRB measure, the volume measure of the saturation of the attractor by stable leaves converges exponentially fast under the dynamics to the unique SRB measure supported on the attractor (the notion of rate of convergence relates to a fixed space of test functions). However, in the general case it is not clear if one can expect to always achieve an SRB measure as a limit point of the pushed Riemannian volume. In particular, some ``nice" systems do not admit an SRB measure (see \cite{YoungCounterExample}).

This gives rise to the natural question: When can we achieve an SRB measure through pushing forwards the Riemannian volume of a smooth dynamical system? Before we describe the results of this paper and how they relate to this question, we wish to mention another fundamental field of studies in smooth ergodic theory, and how it relates to this question. 

The {\em smooth realization problem} posed by von Neumann is the question of what dynamical systems $(X,T,\nu)$ (not necessarily smooth) can be realized through a measure theoretic isomorphism as a smooth system $(M,f,m)$, where $M$ is a closed Riemannian manifold, $f$ is a smooth diffeomorphism of $M$, and $m=m\circ f^{-1}$ is the Riemannian volume of $M$. Notice that an immediate restriction of the smoothly-realizable dynamical systems is having finite metric entropy. A recent advancement in this direction is due to Dolgopyat, Kanigowski, and Rodriguez-Hertz, where they prove that for smooth systems which preserve volume, exponential mixing implies Bernoulli (\cite{ExpMixBern}). Exponential mixing is a property of the smooth structure, as it requires specifying a space of regular test functions on which the mixing estimates hold; however their result nonetheless explores a restriction on the ergodic properties of smooth systems. See also \cite{PesinSenti} and the corresponding discussion in \cite{KatokBook23}.

A natural extension of the smooth realization problem can then be, what dynamical systems $(X,T,\nu)$ can be realized through a measure theoretic isomorphism as a smooth system $(M,f,\mu)$, where $M$ is a closed Riemannian manifold, $f$ is a smooth diffeomorphism of $M$, and $\mu=\lim_n m\circ f^{-n}$, where $m$ is the Riemannian volume of $M$. Similarly, $\int g\circ f^n h dm\xrightarrow[]{\mathrm{exp}} \int gdm\int hdm$ when $m=m\circ f^{-1}$, can be naturally extended to  $\int g\circ f^n h dm\xrightarrow[]{\mathrm{exp}} \int gd\mu\int hdm$ where $m$ is not necessarily invariant, but $\mu$ is. Can we say that $\mu$ is Bernoulli in that case? We believe that the answer is positive based on a consequence of this work, as we explain in \textsection \ref{mainResults}. 

The problem of finding a Banach space of test functions which admits certain properties is not a trivial issue. Another instance of that same challenge is proving the spectral gap property, which requires defining a suitable Banach space of test functions on which the dynamics act as a linear operator with a spectral gap. Often the space of such test functions is non-trivial in the sense that one studies functions which are regular on stable leaves, but may have merely measurable behavior w.r.t the topology of the ambient manifold.

Moreover, the relationship between properties such as a spectral gap (on some ``reasonable" Banach space) and exponential mixing is still an open mystery. In what cases can one have exponential mixing without a spectral gap? These types of questions are generally still open, while being fundamental. 

Finally, an additional natural property in this family would be the exponential convergence of the volume to an invariant measure, as in \eqref{expmix}. This property on its own is not enough to conclude any stronger ergodic properties (e.g $f=\mathrm{Id}_M$, or even $f=A\times \mathrm{Id}_{\mathbb{S}^1}$ where $A$ is a volume-preserving linear Anosov map of the torus). However, we show that it is indeed enough to conclude that the limiting measure is an SRB measure, possibly in the degenerate sense that $h_\mu(f)=\int \sum \chi^+d\mu=0$.

To sum up the two independent directions of study we mentioned: We wish to understand when can limits of the pushed volume be SRB measures for thermodynamic purposes; and also we wish to understand what properties restrict smooth systems in terms of ergodic properties and the extended smooth realization problem. Possible future lines of study include exploring the relationship between different smooth ergodic properties, such as exponential mixing and a spectral gap. 

\subsection{Main results}\label{mainResults} $M$ is a $d$-dimensional closed Riemannian manifold, $f\in\mathrm{Diff}^{1+\beta}(M)$, and $m$ denotes the Riemannian volume form of $M$. Let us first introduce three different notions of {\em exponential convergence}: $\exists C>0,\alpha\in(0,1],\gamma>0$ s.t $\forall g,h\in \mathrm{H\ddot{o}l}_\alpha(M)$,

\begin{equation}\label{expmixPrime}\left|\frac{1}{N}\sum_{k=n}^{n+N-1}m(g\circ f^k)-\mu(g)\right|\leq C\cdot \|g\|_\alpha\cdot e^{-\gamma \cdot\min\{n, N\}},\end{equation}

\begin{equation}\label{expConv}\Big|m(g\circ f^n)-\mu(g)\Big|\leq C\cdot \|g\|_\alpha\cdot e^{-\gamma n},\end{equation}

\begin{equation}\label{trueExpMix}\left|\int g\circ f^n\cdot hdm-\int gd\mu\int hdm \right|\leq C\cdot \|g\|_\alpha\|h\|_\alpha\cdot e^{-\gamma n}.\end{equation} 

It is clear that \eqref{trueExpMix}$\Rightarrow$\eqref{expConv}$\Rightarrow$\eqref{expmixPrime}. We also say that {\em the volume is almost exponentially mixing} (however the volume need \textbf{not} be $f$-invariant) if
\begin{equation}\label{expDecay}
	\left|\int g\circ f^n\cdot hdm\right|\leq C\cdot \|g\|_\alpha\|h\|_\alpha\cdot e^{-\gamma n}\text{, whenever }\int hdm=0.
\end{equation}
Note that \eqref{expDecay} is proper even when $m$ is not $f$-invariant; but when $m=m\circ f^{-1}$ it is equivalent to $m$ being exponentially mixing. The advantage of condition \eqref{expDecay} is that it does not require a-priori a background $f$-invariant measure in order to test it.

\medskip
The main results of the manuscript are structured in the following way:
\begin{enumerate}
	\item In \textsection \ref{ergoCase} we show that exponential convergence in the sense of \eqref{expmixPrime} to an ergodic limit point implies that the limit point is an SRB measure (not necessarily with a positive entropy). The purpose of this section is didactic. We show in addition that in this case where $\mu$ is ergodic, it is also a {\em weakly physical} measure (see Definition \ref{WPhys}) with a full Basin (see Theorem \ref{PHYSICS}).
\item In \textsection \ref{NonErgoCase} we prove that exponential convergence in the sense of \eqref{expmixPrime} to a limit point (not necessarily ergodic) implies that the limit point is an SRB measure (still not necessarily with positive entropy). Note that one cannot expect more ergodic properties without additional assumptions.
	\item In \textsection \ref{almostExpMix} we show that almost exponential mixing of the volume in the sense of \eqref{expDecay} implies the existence of an $f$-invariant $\mu$ such that $m\circ f^{-n}\xrightarrow[n\to\infty]{}\mu$ exponentially fast (in the strong sense of \eqref{trueExpMix}), and consequently that condition \eqref{trueExpMix} implies that $\mu$ must either be the unique (and hence ergodic) SRB measure of the system, and has positive entropy%, and it also admits certain mixing properties on unstable leaves (see Proposition \ref{UMix})
, or that $\mu$ is a Dirac mass at a fixed point which is an SRB measure in the degenerate sense that $h_\mu(f)=\int \sum \chi^+d\mu=0$. Note that the degenerate case cannot be ruled out, as illustrated in the remark after Theorem \ref{posExps}. %Moreover, in Theorem \ref{posExps} we show a uniform bound from below on the positive exponents of all ergodic invariant measures of the system, aside for at most possibly the limit point $\mu$.
\end{enumerate}

In the case we treat in \textsection \ref{almostExpMix}, when $h_\mu(f)>0$, we believe that the methods of \cite{ExpMixBern} can be extended to show that $\mu$ is Bernoulli. This is a consequence of the observation that the proof of \cite{ExpMixBern} only truly requires the conditional measures on unstable leaves to be smooth, and Proposition \ref{UMix} gives the right notion of exponential mixing on unstable leaves for their methods to be extended. 

In addition, notice that the assumption of \eqref{expmixPrime} is formally weaker than \eqref{expConv}. The weaker assumption allows one to rely on some averaging in order to gain exponential convergence, rather than just pushing forwards the volume.

\medskip
Our proof relies on the following tools: We use coverings by exponential Bowen balls of the form $B(\cdot,n,e^{-n\delta})$ which have the following three properties:
\begin{enumerate}
	\item $\lim_{\delta\to0}\limsup\frac{-1}{n}\log\mu(B(\cdot,n,e^{-n\delta}))=h_{\mu_x}(f)$ $\mu$-a.e, where $\mu=\int \mu_xd\mu(x)$ is its ergodic decomposition (see \cite{NLE}),
	\item If $x$ is a Pesin regular point, then for all $n$ large enough, $\forall k\leq n$, $f^k[B(x,n,e^{-n\delta})] $ is contained in the Pesin chart of $f^k(x)$,
	\item Subsets of Pesin blocks can be covered by exponential Bowen balls with exponentially low multiplicity, for all $n$ large enough (see \cite[Lemma~2.2]{NLE}).
\end{enumerate}

Furthermore, our proof of the results of \textsection \ref{almostExpMix} relies on the construction of fake $cs$-foliations which are absolutely continuous in small exponential neighborhoods of Pesin regular points. These fake foliations were constructed by Dolgopyat, Kanigowski, and Rodriguez-Hertz in \cite{ExpMixBern}.	

The key idea of the proof of \textsection \ref{ergoCase} and \textsection \ref{NonErgoCase} is a type of shadowing argument, where since we cannot mix on exponential Bowen balls of $n$ steps, we break down the orbit segment of $n$ steps into $\frac{1}{\epsilon}$-many orbit segments of $n\epsilon$-many steps. Thus we can study points which remain close to a large measure set of ``good points", but not necessarily lie in the Bowen ball of any ``good point" for the whole $n$ steps.

\tableofcontents

\section{The ergodic case}\label{ergoCase}

\subsection{Preliminary parameter choices}\label{prel}
For didactic purposes, we treat first the ergodic case, as the argument is much clearer in that case. In this section (and in \textsection \ref{NonErgoCase}) we assume that $\exists C,\gamma,\alpha>0$ s.t 
\begin{equation}\label{expmix} \forall g\in \mathrm{H\ddot{o}l}_\alpha(M),\text{ }\left|\frac{1}{N}\sum_{k=n}^{n+N-1}m(g\circ f^k)-\mu(g)\right|\leq C\cdot \|g\|_\alpha\cdot e^{-\gamma \cdot\min\{n, N\}},\end{equation}
the weakest notion of exponential convergence. Assume that $\mu$ is ergodic%, and that $\chi^u(\mu)>h_\mu(f)$, otherwise we are done
. Let $\epsilon%\ll \frac{\chi^u(\mu)-h_\mu(f)}{2}%\equiv \Delta_\mu
>0$, and set:  %By Lemma \ref{chis}, and since $\chi^u(\mu)>0$, it follows 
%%Assume w.l.o.g that $\chi^s(\mu)>0$ (otherwise $B(x,e^{-\epsilon n})\subseteq B(x,-n,e^{-\frac{\epsilon}{2}n})$ for $\mu$-a.e $x$ and sufficiently large $n$, and we can apply the exponential mixing assumption on the exponential Bowen ball). Therefore the following is well defined.

\begin{enumerate}
\item Let $\mathcal{K}_\epsilon$ be a set s.t $\mu(\mathcal{K}_\epsilon)\geq e^{-\epsilon^4}$, $\mu(B(x,-n\epsilon,e^{-2\delta n})),\mu(B(x,-n\epsilon,e^{-\delta n}))=e^{-n\epsilon(h_\mu(f)\pm \epsilon^2)}$ for all $n\geq n_\epsilon$, for some $\delta\in (0,\epsilon^2)$ (see \cite{NLE}).
%Brin-Katok set where the measure estimate for Bowen balls of positive radius $\leq r_\epsilon\leq \frac{1}{\ell}$ is bounded by an $\epsilon^2$ estimates for all $n\geq n_\epsilon\in\mathbb{N}$, and $\mu(\mathcal{K}_\epsilon)\geq e^{-\epsilon ^{4}}$. %Furthermore, we want that $$\frac{\mu(B_{2^{\frac{\chi^s(\mu)}{h_\mu(f)}}r}(x))}{\mu(B_{r}(x))}=2\cdot e^{\pm\epsilon^2},\forall x\in\mathcal{K}_\epsilon,$$
%and sufficiently small $r>0$. We assume w.l.o.g that $n_\epsilon=0$.

%\item Let $n_\epsilon'\in\mathbb{N}$ and a corr. Lebesgue Differentiation set $\mathcal{L}_\epsilon\subseteq \Lambda_{\ell'}\cap \mathcal{K}_\epsilon$ (w.r.t the $s$-measures of $\mu|_{\Lambda_{\ell'}}$) s.t
%\begin{enumerate}
%\item $\mu(\mathcal{L}_\epsilon)\geq \mu(\Lambda_{\ell'}\cap \mathcal{K}_\epsilon)e^{-\epsilon^{3\cdot2^\kappa}}$,
%\item $n\geq n_\epsilon', x\in\mathcal{L}_\epsilon$, $\mu_{B_{r_\epsilon}(x,n)}(\Lambda_{\ell'}\cap \mathcal{K}_\epsilon)\geq e^{-\epsilon^{3\cdot2^\kappa}}$.
%\end{enumerate}
\item Let $\ell=\ell(\epsilon)\in\mathbb{N}$ s.t $\mu(\Lambda_{\ell}^{(\underline{\chi},\tau)})\geq e^{-\epsilon^{4}}$, with $0<\tau<\min\{\tau_{\underline\chi},\frac{1}{3d}\delta^3\}$.
\item Set $E_\epsilon:= \Lambda_\ell\cap \mathcal{K}_\epsilon$. Then $\mu(E_\epsilon)\geq  e^{-\epsilon ^{3}}$ for all sufficiently small $\epsilon>0$. W.l.o.g assume that $\epsilon=\frac{1}{p^2}$, and that $p^2| n$ when we choose some large $n$ s.t $e^{-\delta n}\ll\frac{1}{\ell}$, so the ceiling values can be omitted.\footnote{That is, $n$ takes values in $\{m\cdot p^2\}_{m\geq 1}$ for some $p\in\mathbb{N}$.}
%\item Let $\mathcal{L}_\epsilon$ be the set of Lebesgue Differenitation density points of $\mathcal{K}_\epsilon\cap \Lambda_{\ell}$, w.r.t $\mu$, \footnote{By applying the proof to functions which are constant on fake $cu$-leaves in the $r_\epsilon$ chart of the point, we want narrow strips around the point instead of balls.} s.t $\mu(B_{r}(x)\cap \mathcal{K}_\epsilon\cap \Lambda_{\ell})\geq e^{-\epsilon^2}\mu(B_{r}(x))$ $\forall x\in \mathcal{L}_\epsilon$, $\forall 0<r\leq s_\epsilon$\footnote{This ball spans over the $cu$-direction of the chart, and has a small $s$-radius.}. Set $n_\epsilon'$ so large that $e^{-n_\epsilon' \epsilon^2}<s_\epsilon$. W.l.o.g assume $n_\epsilon'=0$.
%\item Let $\mathcal{C}^{(n)}$ be a cover of $\Lambda_\ell$ by charts of size $\substack{e^{-\epsilon n}\\s}\times\substack{r_\epsilon\\cu}$ s.t $\#\mathcal{C}^{(n)}\leq C_d\cdot e^{\epsilon n}$, where $C_d$ is the Besicovitch constant of $M$. 
\item Let $n\geq n_\epsilon$, and let $%\widetilde{
\widetilde{\mathcal{A}}%}
^{(n)}_\epsilon$ be a cover of $E_\epsilon$ by %charts of size $\substack{\frac{1}{2}e^{-\epsilon n}r_\epsilon\\s}\times\substack{\frac{1}{2}e^{-\epsilon^2n}r_\epsilon\\cu}$ s.t 
Bowen balls $B(\cdot, -n\epsilon, e^{-2\delta n} )$ with multiplicity bounded by $e^{3d\tau n}\leq e^{\delta^3 n}\leq e^{\epsilon^6n}$, and in particular with cardinality bounded by $\#\mathcal{A}^{(n)}_\epsilon\leq e^{n\epsilon (h_\mu(f)+\epsilon^2)}e^{\epsilon^6n}$, as in \cite[Lemma~2.2]{NLE}. Set $\mathcal{A}^{(n)}_\epsilon:=\{B(x,n\epsilon,e^{-\delta n}): B(x,n\epsilon,e^{-2\delta n})\in \widetilde{\mathcal{A}}^{(n)}_\epsilon\}$. 
\end{enumerate}

\subsection{Large asymptotic-volume set of points shadowed by a set of $\mu$-good points}\label{shadowing}

%\begin{definition}
%$\mathcal{A}^{(n)}_\epsilon:=\{2B:B\in\mathcal{C}^{(n)}_\epsilon\}$.	
%\end{definition}

\begin{lemma}\label{firstLemma}
Let $0\leq i\leq(1-2\sqrt{\epsilon}) \frac{1}{\epsilon}$% s.t $i\cdot \frac{n\epsilon}{\chi^s}> \frac{\gamma}{4}$
. Then for all $\epsilon>0$ sufficiently small, $\exists n_\epsilon'\geq n_\epsilon$ s.t for all $n\geq n_\epsilon'$, $$\frac{1}{\sqrt\epsilon n}\sum_{k=-i\epsilon n +n(1-\sqrt{\epsilon})}^{-i\epsilon n+n-1} m\circ f^{-k}(M\setminus\bigcup\mathcal{A}^{(n)}_\epsilon)\leq \epsilon^2.$$
\end{lemma}
\begin{proof}
$\frac{1}{\sqrt\epsilon n}\sum_{k=-i\epsilon n+ n(1-\sqrt\epsilon)}^{-i\epsilon  n+n-1} m\circ f^{-k}(M\setminus\bigcup\mathcal{A}^{(n)}_\epsilon)=1-\frac{1}{\sqrt\epsilon n}\sum_{k=-i\epsilon n+ n(1-\sqrt\epsilon)}^{-i\epsilon n+n-1} m\circ f^{-k}(\bigcup\mathcal{A}^{(n)}_\epsilon)= 1-\frac{1}{\sqrt\epsilon n}\sum_{k=-i\epsilon n+ n(1-\sqrt\epsilon)}^{-i\epsilon n+n-1} m\circ f^{-k}(\mathbb{1}_{\bigcup\mathcal{A}^{(n)}_\epsilon})$.

For every $B\in\mathcal{C}^{(n)}_\epsilon$, define $g^{(n)}_B$ be a Lipschitz function s.t $g^{(n)}_B|_{B(x_B,n\epsilon,e^{-2\delta n})}=1$, $g^{(n)}_B|_{B(x_B,n\epsilon,e^{-\delta n})^c}=0$, and $\mathrm{Lip}(g^{(n)})\leq e^{2n\epsilon\log M_f}$, where $M_f:=\max_M\{\|d_\cdot f\|, \|d_\cdot f^{-1}\|\}$. %\footnote{$e^{-\epsilon n}-e^{-\epsilon n}e^{-e^{-2\epsilon nd}}= e^{-\epsilon n}(1-e^{-e^{-2\epsilon nd}} )\geq \frac{1}{2}e^{-\epsilon n}\cdot e^{-2\epsilon nd}\geq\frac{1}{2e^{3\epsilon nd}}$.} 
Notice: $\mathbb{1}_{\bigcup \mathcal{A}^{(n)}_\epsilon}=\max_{B\in\mathcal{A}^{(n)}_\epsilon}\mathbb{1}_B\geq \max_{B\in\mathcal{C}_\epsilon^{(n)}}g^{(n)}_B =: g^{(n)}$.

\textit{Claim:} $\mathrm{Lip}(g^{(n)})\leq e^{2n\epsilon\log M_f}$.

\textit{Proof:}
We prove that if $g_1$ and $g_2$ are $L$-Lipschitz, then $g_1\vee g_2:=\max\{g_1,g_2\}$ is $L$-Lipschitz. The claim for $g^{(n)}$ follows by induction. Let $x,y\in M$. If $g_1(x)\geq g_2(x)$ and $g_1(y)\geq g_2(y)$, or if $g_2(x)\geq g_1(x)$ and $g_2(y)\geq g_1(y)$, then
$$\frac{| (g_1\vee g_2)(x) -(g_1\vee g_2)(y) |}{|x-y|}\leq L$$
by the Lipschitz properties of $g_1$ and $g_2$.

We therefore may assume that w.l.o.g $g_1(x)\geq g_2(x)$ and $g_1(y)\leq g_2(y)$ (otherwise switch the roles of $g_1$ and $g_2$). Then,

$$g_1(x)\leq L\cdot|x-y|+g_1(y)\leq L\cdot|x-y|+g_2(y),$$
and so
$$g_1(x)-g_2(y)\leq L\cdot |x-y|.$$
Similarly,

$$g_2(y)-g_1(x)\leq L\cdot |x-y|.$$

Therefore, $| g_1(y)-g_2(x)|\leq L\cdot |x-y| $, and so

$$| (g_1\vee g_2)(x) -(g_1\vee g_2)(y)|=|g_1(x)-g_2(y)|\leq L\cdot |x-y|.$$
This concludes the proof of the claim.

\medskip
By the exponential convergence of the volume averages given by \eqref{expmix}, 

\begin{align*}
\frac{1}{\sqrt\epsilon n}\sum_{k=-i\epsilon n+ n(1-\sqrt\epsilon)}^{-i\epsilon  n+n-1} m\circ f^{-k}(\mathbb{1}_{\bigcup\mathcal{A}^{(n)}_\epsilon})\geq & \frac{1}{\sqrt\epsilon n}\sum_{k=-i\epsilon n+ n(1-\sqrt\epsilon)}^{-i\epsilon  n+n-1}  m\circ f^{-k}(g^{(n)})=\mu(g^{(n)})\pm Ce^{-\gamma \sqrt\epsilon n}e^{2\epsilon\log M_f n}\\
\geq &\mu(\max_{B\in\widetilde{\mathcal{A}}^{(n)}_\epsilon}\mathbb{1}_B)- Ce^{-\gamma \sqrt\epsilon n}e^{2\epsilon\log M_f n} =\mu(\bigcup \widetilde{\mathcal{A}}^{(n)}_\epsilon)- Ce^{-\gamma \sqrt\epsilon n}e^{2\epsilon\log M_f n}\\
\geq&\mu(E_\epsilon)- Ce^{-\gamma \sqrt\epsilon n}e^{2\epsilon\log M_f n}\geq e^{-\epsilon^2}-Ce^{-\gamma \sqrt\epsilon n}e^{2\epsilon\log M_f n}.	
\end{align*}

Then for all $\epsilon>0$ s.t $\gamma\sqrt\epsilon>2\log M_f\epsilon$ and $\frac{\epsilon^4}{2}\geq2\frac{\epsilon^6}{6} $, and for all $n$ large enough so $Ce^{-\gamma \sqrt\epsilon n}e^{2\epsilon\log M_f n}\leq \frac{\epsilon^6}{6}$, we have 

\begin{align*}
\frac{1}{\sqrt\epsilon n}\sum_{k=-i\epsilon n+ n(1-\sqrt\epsilon)}^{-i\epsilon  n+n-1}  m\circ f^{-k}(M\setminus\bigcup\mathcal{A}^{(n)}_\epsilon)\leq 1-e^{-\epsilon^2}+Ce^{-\gamma \sqrt\epsilon n}e^{2\epsilon n\log M_f}\leq \epsilon^2-\frac{\epsilon^4}{2}+\frac{\epsilon^6}{6} +\frac{\epsilon^6}{6}\leq \epsilon^2.	
\end{align*}

\end{proof}

\begin{definition}\label{Sn}
$$\mathcal{S}_n:=\{x\in \bigcup\mathcal{A}^{(n)}_\epsilon:\text{for all }0\leq i\leq(1-2\sqrt{\epsilon})\frac{1}{\epsilon}, f^{-i n\epsilon}(x)\in\bigcup\mathcal{A}^{(n)}_\epsilon\}$$ is the set of points in $\bigcup\mathcal{A}^{(n)}_\epsilon$ which are shadowed by $E_\epsilon$ for at least $(1-2\sqrt\epsilon)n$-many steps backwards.	
\end{definition}

%Notice, the elements of $\mathcal{A}^{(n)}_\epsilon$ are contained in $\frac{n\epsilon}{\chi^s}$-Bowen balls of points in $E_\epsilon${\color{blue} (up to a small error which depends on the preciseness of Pesin estimates in charts)}{\color{green} this disclaimer is no longer needed when we use exp. Bowen balls and not charts}.

\begin{theorem}\label{firstTheorem}
Let $n'_\epsilon\geq0$ as in Lemma \ref{firstLemma}, then for all $\epsilon>0$ sufficiently small and for all $n\geq n_\epsilon'$, 
$$\frac{1}{\sqrt\epsilon n}\sum_{k=n(1-\sqrt\epsilon)}^{n-1}  m\circ f^{-k}(\mathcal{S}_n)\geq e^{-\epsilon^\frac{3}{4}
}.$$
\end{theorem}

\begin{proof}

Let $n\geq n_\epsilon'$. Let $B\in\mathcal{C}^{(n)}_\epsilon$. We break down the pull-back of $B$ as follows: $f^{-n}[B]=f^{-n\epsilon}\circ\cdots\circ f^{-n\epsilon}[B]$, where the composition chain has $%(1-\sqrt\epsilon)
\frac{1}{\epsilon}$ many steps.

\begin{align*}
\frac{1}{\sqrt\epsilon n}\sum_{k= n(1-\sqrt\epsilon)}^{n-1}  m\circ f^{-k}(\mathcal{S}_n)=& \frac{1}{\sqrt\epsilon n}\sum_{k=n(1-\sqrt\epsilon)}^{n-1} m\circ f^{-k}(\{x\in \bigcup\mathcal{A}^{(n)}_\epsilon:\text{for all }0\leq i\leq(1-2\sqrt{\epsilon})\frac{1}{\epsilon}, f^{-in\epsilon}(x)\in\bigcup\mathcal{A}^{(n)}_\epsilon \})\\
	=& \frac{1}{\sqrt\epsilon n}\sum_{k=n(1-\sqrt\epsilon)}^{n-1}   m\circ f^{-k}\left(\bigcap_{i=0 }^{(1-2\sqrt\epsilon)\frac{1}{\epsilon}} f^{in\epsilon}[\bigcup\mathcal{A}^{(n)}_\epsilon]\right)\\
	\geq&1-\sum_{i=0 }^{(1-2\sqrt\epsilon)\frac{1}{\epsilon}} \frac{1}{\sqrt\epsilon n}\sum_{k= n(1-\sqrt\epsilon)}^{n-1}  m\circ f^{-k}(M\setminus f^{in\epsilon}[\bigcup\mathcal{A}^{(n)}_\epsilon])\\
	=&1-\sum_{i=0}^{(1-2\sqrt\epsilon)\frac{1}{\epsilon}} \frac{1}{\sqrt\epsilon n}\sum_{k=-i\epsilon n+ n(1-\sqrt\epsilon)}^{-i\epsilon  n+n-1} m\circ f^{-k}(M\setminus\bigcup\mathcal{A}^{(n)}_\epsilon)\\
	\geq&1-(\frac{1-2\sqrt\epsilon}{\epsilon}+1)\cdot \epsilon^2 \text{ }(\because \text{Lemma }\ref{firstLemma}).
\end{align*}

For all $\epsilon>0$ small enough so $1-\epsilon-\epsilon^2\geq e^{-\epsilon^\frac{3}{4}}$, the theorem follows.
\end{proof}

\subsection{A cover by exponential Bowen balls via concatenation}

\begin{definition}
	Let $\mathcal{C}_\epsilon^{(n)}:=\bigvee_{i=0}^{\frac{1-2\sqrt\epsilon}{\epsilon}}f^{in\epsilon}[\mathcal{A}_\epsilon^{(n)}]$.
\end{definition}

\noindent\textbf{Remark:} Notice that $\mathcal{C}_\epsilon^{(n)}$ covers $\mathcal{S}_n$, and that $\#\mathcal{C}_\epsilon^{(n)}\leq (\#\mathcal{A}_\epsilon^{(n)})^\frac{1}{\epsilon}\leq e^{\frac{1}{\epsilon}(n\epsilon h_\mu(f)+\epsilon^{2.5}n)}\leq e^{nh_\mu(f)+\epsilon^\frac{3}{2}n}$.

\begin{lemma}\label{tubeVol}
	Let $B\in\mathcal{C}_\epsilon^{(n)}$, then for all $n$ large enough and $\epsilon>0$ small enough (independently of $B$), $\frac{1}{\sqrt\epsilon}\sum_{k=(1-\sqrt\epsilon)n}^{n-1}m\circ f^{-k}(B)\leq  e^{-\chi^u n+2\epsilon^\frac{1}{3} n}$, where $e^{-\chi^un}:=\prod_{\chi_i>0}e^{-\chi^u_i n}$.
\end{lemma}
\begin{proof}
Let $y\in B$, then $B\subseteq B(y,-n(1-2\sqrt\epsilon),2e^{-\delta n})$. For every $k\in[(1-\sqrt\epsilon n),n-1]$, $m\circ f^{-k}[B]\leq e^{2d\sqrt\epsilon n\log M_f}m(B(f^{-n(1-2\sqrt{\epsilon})}(y),n(1-2\sqrt\epsilon),2e^{-\delta n}))$. We show that $m(B(f^{-n(1-2\sqrt{\epsilon})}(y),n(1-2\sqrt\epsilon),2e^{-\delta n}))\leq e^{-(\chi^u -\epsilon)n+2\epsilon^\frac{1}{3}n}$. 

Write $x:= f^{-n(1-2\sqrt{\epsilon})}(y) $. Let $x_i$ s.t $f^{n\epsilon i}(x)\in B(x_i,n\epsilon,2e^{-\delta n})$, $x_i\in f^{-n\epsilon}[E_\epsilon]$, $0\leq i\leq \frac{1-2\sqrt{\epsilon}}{\epsilon}$.

Assume for contradiction that there exists a volume form $\omega^u_0\in \wedge^{\mathrm{dim}H^u(\mu)}T_yM$ s.t $|\omega^u_0|\geq e^{-(\chi^u-\epsilon)n+\epsilon ^\frac{1}{3} n}$, and s.t $\sphericalangle(d_{x}\psi_{x_0}^{-1}\omega^u_0,E^u)\leq \epsilon^2$, where $E^u$ is the unstable direction of $x_0$ in its Pesin chart $\psi_{x_0}$; and finally assume that $\exp_x\omega^u_0\subseteq f^{-n(1-2\sqrt{\epsilon})}[B]$ (when we think of $\omega^u_0$ as the parallelogram it defines in $T_xM$). We will show a contradiction by showing that $f^{n(1-2\sqrt\epsilon)}[\exp_x\omega^u_0] $ contains a geodesic of length greater than $2e^{-\delta n}$, which contradicts $B(f^{n(1-2\sqrt\epsilon)}(x),n(1-2\sqrt\epsilon),2e^{-\delta n})\supseteq B\supseteq f^{n(1-2\sqrt\epsilon)}[\exp_x\omega^u_0] $.
  
The choice of $\omega_0^u$ implies that $|d_xf^{n\epsilon}\omega_0^u|\geq e^{-(\chi^u-\epsilon)n(1-\epsilon)+\epsilon ^\frac{1}{3} n}(1-\epsilon)$. Note, since $f^{n\epsilon}(x_0),f^{n\epsilon}(x_1)\in \Lambda^{(\underline\chi,\tau)}_\ell$, we have $f^{n\epsilon}(x_0),x_1\in \Lambda^{(\underline\chi,\tau)}_{e^{\tau n\epsilon}\ell}$, while also having $2e^{-\delta n}\ll e^{-\tau n\epsilon}\frac{1}{\ell}$ for all $n$ large enough. By the H\"older continuity of the unstable spaces of points in $ \Lambda^{(\underline\chi,\tau)}_{e^{\tau n\epsilon}\ell}$ (\cite[Appendix~A]{BrinHolderContFoliations}), $d_xf^{n\epsilon}\omega^u_0$ projects to $\omega^u_1\in \wedge^{\mathrm{dim}H^u(\mu)}T_{f^{n\epsilon}(x)}M $ s.t $\sphericalangle(d_{x}\psi_{x_1}^{-1}\omega^u_1,E^u)\leq \epsilon^2$ and s.t $|\omega_1^u|\geq (1-\epsilon)\cdot e^{-(\chi^u-\epsilon)n(1-\epsilon)+\epsilon ^\frac{1}{3} n}(1-\epsilon) $.

Continuing by induction, let $\omega^u_{\frac{1-2\sqrt\epsilon}{\epsilon}}$ be a component of $d_xf^{(1-2\sqrt\epsilon)n}\omega_0^u$ s.t $|\omega^u_{\frac{1-2\sqrt{\epsilon}}{\epsilon}}|\geq e^{\epsilon ^\frac{1}{3} n-2\sqrt\epsilon(\chi^u-\epsilon)n}(1-\epsilon)^\frac{1-2\sqrt\epsilon}{\epsilon}%\gg 1
$.

Now, for all $\epsilon>0$ small enough so $2d\sqrt{\epsilon}\log M_f+\epsilon>\epsilon^\frac{1}{3}$, we have for all $n$ large enough, $|\omega^u_{\frac{1-2\sqrt{\epsilon}}{\epsilon}}|\geq 2^de^{-\delta n}$, and we are done.  
\end{proof}

\begin{cor}
	$h_\mu(f)=\chi^u=:\sum_{\chi_i(\mu)>0}\chi_i(\mu)$.
\end{cor}
\begin{proof}
	Let $\epsilon>0$ small enough and $n$ large enough for Theorem \ref{firstTheorem} and Lemma \ref{tubeVol}. Then,
\begin{align*}
		e^{-\epsilon^\frac{3}{4}}\leq &\frac{1}{\sqrt{\epsilon}n}\sum_{k=(1-\sqrt\epsilon)n}^{n-1}m\circ f^{-k}(\mathcal{S}_n)\leq \frac{1}{\sqrt{\epsilon}n}\sum_{k=(1-\sqrt\epsilon)n}^{n-1}m\circ f^{-k}(\bigcup\mathcal{C}_\epsilon^{(n)})\\
		\leq & \#\mathcal{C}_\epsilon^{(n)}\cdot \max_{B\in\mathcal{C}_\epsilon^{(n)}} \frac{1}{\sqrt{\epsilon}n}\sum_{k=(1-\sqrt\epsilon)n}^{n-1}m\circ f^{-k}(B)\leq e^{nh_\mu(f)+\epsilon^\frac{3}{2}n}\cdot e^{-\chi^un+2\epsilon^\frac{1}{3}n}.
\end{align*}
Whence since this holds for arbitrarily large $n$, $h_\mu(f)\geq \chi^u-3\epsilon^\frac{1}{3}$. Since $\epsilon>0$ is arbitrarily small, and by the Ruelle inequality, $h_\mu(f)=\chi^u$.
\end{proof}

\subsection{Physicality}

Given a subsequence $(n_k)_k\subseteq\mathbb{N}$, we denote by $\overline{d}((n_k)_k)$ the upper-density $\limsup_N\frac{1}{N}\#\{k: n_k\leq N\}$.
\begin{definition}\label{WPhys}
A Borel probability measure $\nu$ is called {\em weakly physical} with a basin $$B_\nu:=\left\{y\in M: \exists n_k\uparrow\infty\text{ with }\overline{d}((n_k)_k)=1\text{ s.t } \forall g\in C(M),\frac{1}{n_k}\sum_{j=0}^{n_k-1}g\circ f^j(y)\rightarrow \int gd\nu\right\}$$
if $m(B_\nu)>0$. If $m(B_\nu)=1$, we say that $\nu$ is {\em weakly physical with a full basin}.
\end{definition}

\begin{theorem}[Weak physicality]\label{PHYSICS}
The measure $\mu$ is weakly physical with a full basin.
\end{theorem}
\begin{proof}
The idea of the proof is similar to how we prove the entropy formula for $\mu$, where we show that the volume of points which shadow ``good" $\mu$-points is large, but we add an additional restriction to the set of ``good" $\mu$-points asking them to be $\mu$-generic; and show that trajectories which shadow them inherit this property.

\medskip
Since $C(M)$ is separable, there exists a countable set $\mathcal{L}(M)\subseteq C(M)$ s.t $\forall h\in C(M)$, for all $\tau>0$, there exists $g\in \mathcal{L}(M)$ s.t $\|g-h\|_\infty\leq \tau$.

Let $\epsilon>0$, and for $g\in \mathcal{L}(M)$ let $L(\cdot)$ be a monotone uniform-continuity function of $g$.\footnote{That is, $\forall \widetilde{\epsilon}>0$, $d(x,y)\leq L(\widetilde{\epsilon})\Rightarrow |g(x)-g(y)|\leq \widetilde{\epsilon}$.} Write $\mathcal{L}(M)=\{g_i\}_{i\geq0}$, and set $L_\epsilon:=\min_{i\leq\frac{1}{\epsilon}}L_{g_i}$. We may assume w.l.o.g that for all $g\in\mathcal{L}(M)$, $\|g\|_\infty\in (0,1]$.

\begin{comment}
Let $g\in %\mathrm{Lip}_+(M,\mu):=\{g\in 
C(M)%: g\geq  0, \int gd\mu>0\}
$, and let $L(\cdot)$ be uniform-continuity function of $g$.\footnote{That is, $\forall \widetilde{\epsilon}>0$, $d(x,y)\leq L(\widetilde{\epsilon})\Rightarrow |g(x)-g(y)|\leq \widetilde{\epsilon}$.} Let $\epsilon>0$ arbitrary small.	W.l.o.g $\|g\|_\infty>0$, o.w every $y\in M$ is $g$-generic.
\end{comment}
	
\medskip
\textbf{Step 1:} Further restricted set of ``good" $\mu$-points.

\textbf{Proof:} 	Consider the set $E_\epsilon$ from \textsection \ref{prel}, and assume further w.l.o.g that for every $x\in E_\epsilon$, for all $n\geq n_\epsilon$,
\begin{equation}\label{firstGen}
	\frac{1}{n}\sum_{k=0}^{n-1}g_i\circ f^k(x)=\int g_id\mu\pm \epsilon^2%\|g\|_\infty
	,\text{ for all }i\leq \frac{1}{\epsilon}.
\end{equation}

\medskip
\textbf{Step 2:} Large volume of points shadowing ``good" $\mu$-points.

\textbf{Proof:} Recall the definition of the set $\mathcal{S}_n$ from Definition \ref{Sn}, and recall Theorem \ref{firstTheorem}. Notice, for every $n\geq (n_\epsilon\frac{1}{\epsilon^2}\cdot\frac{1}{1-2\sqrt\epsilon})^3$, for every $\epsilon^\frac{1}{3}n\leq N\leq n(1-2\sqrt\epsilon)$, for every $y\in f^{-n}[\mathcal{S}_n]$, for all  $\epsilon>0$ small %w.r.t $g$ 
and $n$ large w.r.t %$g$ and 
$\epsilon$, for all $g\in\{g_i\}_{i\leq \frac{1}{\epsilon}}$,
\begin{equation}\label{2ndGen}
	\frac{1}{N}\sum_{k=0}^{N-1}g\circ f^k(y)=\int gd\mu\pm (L_\epsilon^{-1}(e^{-\delta n})+3\epsilon^\frac{1}{6}\|g\|_\infty)= \int gd\mu\pm \epsilon^\frac{1}{7}.
\end{equation}

Set $$\mathcal{T}_\epsilon^n%(g)
:=f^{-n}[\mathcal{S}_n]\text{ and }\mathcal{T}_\epsilon%(g)
:=\limsup_n \mathcal{T}_\epsilon^n%(g)
.$$
Notice that $m(\mathcal{T}_\epsilon%(g)
)=\lim_n m(\bigcup_{n'\geq n}\mathcal{T}_{\epsilon}^{n'}%(g)
)\geq e^{-\epsilon^\frac{3}{4}}$, and for every $y\in \mathcal{T}_\epsilon%(g)
$, there exists $(n_k^\epsilon)_k\subseteq \mathbb{N}$ with $\overline{d}((n_k^\epsilon)_k)\geq 1-\epsilon^\frac{1}{3}$ s.t for all $k$, for all $g\in \{g_i\}_{i\leq \frac{1}{\epsilon}}$,
\begin{equation}\label{3rdGen}
	\frac{1}{n_k^\epsilon}\sum_{j=0}^{n_k^\epsilon-1}g\circ f^j(y)=\int gd\mu\pm \epsilon^\frac{1}{7}%,
	.
\end{equation}	
%	then 
%\begin{equation}\label{4thGen}
%\int gd\mu-\epsilon^\frac{1}{7}\leq 	\liminf\frac{1}{n}\sum_{k=0}^{n-1}g\circ f^k(y)\leq \limsup\frac{1}{n}\sum_{k=0}^{n-1}g\circ f^k(y)=\int gd\mu+\epsilon^\frac{1}{7}.
%\end{equation}		

\medskip
\textbf{Step 3:} Full volume of %$g$-
future-generic points on an upper-density $1$ subseueqnce.

\textbf{Proof:} Let $\epsilon_\ell\downarrow0$, and let $\widetilde{\mathcal{T}}%^{\ell_0}%(g)
:=\limsup\limits_{\ell%\geq \ell_0
} \mathcal{T}_{\epsilon_\ell}%(g)
$. In particular, $m(\mathcal{T}%^{\ell_0}%(g)
)=\lim_{\ell_0}m(\bigcup\limits_{\ell\geq \ell_0} \mathcal{T}_{\epsilon_\ell})\geq \liminf\limits_{\ell_0}e^{-\epsilon_{\ell_0}^\frac{3}{4}}=1$.

Let $y\in \widetilde{\mathcal{T}}$ and define recursively a sequence in the following way: Let $\epsilon_{\ell_i}\downarrow 0$ s.t $y\in \cap_{i\geq 0}\mathcal{T}_{\epsilon_{\ell_i}}$. Let $(n_k)_{k=0}^{N_0}:=(\lceil\epsilon_{\ell_0}^\frac{1}{3}n_0^{\epsilon_{\ell_0}}\rceil,\ldots, n_0^{\epsilon_{\ell_0}})$; now assume we have $(n_k)_{k=0}^{N_{i}}$, let $k_{\ell_{i+1}}$ be the first $k$ s.t $\epsilon_{\ell_{i+1}}^\frac{1}{3}n_{k}^{\epsilon_{\ell_{i+1}}}> n_{N_i}$, and add $(\lceil\epsilon_{\ell_{i+1}}^\frac{1}{3}n_{k_{\ell_{i+1}}}^{\epsilon_{\ell_{i+1}}}\rceil,\ldots, n_{k_{\ell_{i+1}}}^{\epsilon_{\ell+1}})$ to the end of $(n_{k})_{k=0}^{N_\ell}$. In particular, $\overline{d}((n_k%^{\ell_0}
)_k)\geq 1-\epsilon_{\ell_i}^\frac{1}{3}$ for all $\ell_i$, whence $\overline{d}((n_k%^{\ell_0}
)_k)=1$. 

Now, for every $y\in \widetilde{\mathcal{T}}%^{\ell_0}%(g)
$, for every $i \geq %\ell_
0$, for all $k\geq 0$, for all $g\in\mathcal{L}(M)$,
\begin{align}\label{5thGen}
&\int gd\mu-\epsilon^\frac{1}{7}_{\ell_i}\leq 	\liminf\frac{1}{n _k^{\epsilon_{\ell_i}}
}\sum_{j=0}^{n_k^{\epsilon_{\ell_i}}
-1}g\circ f^j(y)\leq \limsup\frac{1}{n _k^{\epsilon_{\ell_i}}
}\sum_{j=0}^{n_k^{\epsilon_{\ell_i}}
-1}g\circ f^j(y)=\int gd\mu+\epsilon^\frac{1}{7}_{\ell_i},\nonumber\\
\Rightarrow& \frac{1}{n _k%^{(\ell_0)}
}\sum_{j=0}^{n_k%^{(\ell_0)}
-1}g\circ f^j(y)\xrightarrow[k\to\infty]{}\int gd\mu.
\end{align}	

%Then, since $\forall \ell_0\in\mathbb{N}$, $m(\mathcal{T}^{\ell_0}%(g)
%)\geq e^{-\epsilon_{\ell_0}^\frac{3}{4}} $, 
%$$m\left(\left\{y\in M: \exists n_k\uparrow\infty\text{ with }\overline{d}((n_k)_k)=1\text{ s.t }\forall g\in\mathcal{L}(M),\frac{1}{n_k}\sum_{j=0}^{n_k-1}g\circ f^j(y)\xrightarrow[k\to\infty]{}\int gd\mu\right\}\right)%\geq e^{-\epsilon_{\ell_0}^\frac{3}{4}}
%=1.$$
%and we are done.

\medskip
\textbf{Step 4:} %$$m\left(\left\{y\in M: \forall g\in C(M),  \exists n_k\uparrow\infty\text{ with }\overline{d}((n_k)_k)=1\text{ s.t }\frac{1}{n_k}\sum_{j=0}^{n_k-1}g\circ f^j(y)\xrightarrow[k\to\infty]{}\int gd\mu\right\}\right)=1.$$
$\mu$ is weakly physical with a full basin.

\textbf{Proof:} %Let $C^+(M):=\{h\in C(M): h\geq0\}, \mathrm{Lip}^+(M):=\{g\in \mathrm{Lip}(M): g\geq0\} $. By \cite{LipApproxCont}, $\overline{\mathrm{Lip}(M)}^{C(M)}=C(M)$, and by taking taking $g_+(x):=\max\{g(x),0\}$, it is not hard to see that $\overline{\mathrm{Lip}^+(M)}^{C(M)}=C^+(M)$. Moreover, given $h\in C^+(M)$ and $\delta>0$, we can choose $g_\delta\in \mathrm{Lip}^+(M)$ s.t $\|g_\delta-h\|_\infty\leq \frac{\delta}{2}$, and by considering $\overset{\vee}{g}_\delta(x):= \max\{g_\delta(x),\frac{\delta}{2}\}$, we get that $\|\overset{\vee}{g}_\delta-h\|_\infty\leq \delta$, $\overset{\vee}{g}_\delta \geq 0$, $\overset{\vee}{g}_\delta \in \mathrm{Lip}(M)$, and $\int \overset{\vee}{g}_\delta d\mu\geq \delta>0$; whence $\overline{\mathrm{Lip}_+(M,\mu)}^{C(M)}=C^+(M)$.
%
%\medskip
\begin{comment}
Since %$\overline{\mathrm{Lip}%_+
%(M%,\mu
%)}^{C(M)}=C(M)$ (\cite{LipApproxCont}), and 
$C(M)$ is separable, there exists a countable set $\mathcal{L}(M)\subseteq C%_+
(M%,\mu
)$ s.t $\forall h\in C%^+
(M)$, for all $\tau>0$, there exists $g\in \mathcal{L}(M)$ s.t $\|g-h\|_\infty\leq \tau$. 
\end{comment}
Let $$\mathcal{T}%(g)
:= \left\{y\in M:  \exists n_k\uparrow\infty\text{ with }\overline{d}((n_k)_k)=1\text{ s.t }\forall g\in\mathcal{L}(M),\frac{1}{n_k}\sum_{j=0}^{n_k-1}g\circ f^j(y)\xrightarrow[k\to\infty]{}\int gd\mu\right\},$$ 
%and $\mathcal{T}:=\bigcap_{g\in\mathcal{L}(M)}\mathcal{T}(g)$, 
then $\widetilde{\mathcal{T}}\subseteq \mathcal{T}$ and so $m(\mathcal{T})=1$. %and for every $y\in \mathcal{T}$,
%$$\limsup\frac{1}{n}\sum_{k=0}^{n-1}h\circ f^k(y)\leq \limsup\frac{1}{n}\sum_{k=0}^{n-1}g\circ f^k(y)+\tau\leq \int gd\mu+\tau \leq \int hd\mu+2\tau.$$
%The liminf is bounded similarly. Since $\tau>0$ is arbitrary, we are done.
%$$m\left(\left\{y\in M: \forall h\in C%^+
%(M), \frac{1}{n}\sum_{k=0}^{n-1}h\circ f^k(y)\xrightarrow[n\to\infty]{}\int hd\mu\right\}\right)=1.$$
%Since every $h\in C(M)$ can be decomposed into $h=\frac{|h|+h}{2}-\frac{|h|-h}{2}$ with $\frac{|h|+h}{2} , \frac{|h|-h}{2} \in C^+(M)$, we are done. 
Given $y\in \mathcal{T}$, $\tau>0$, and $h\in C(M)$, let $g\in\mathcal{L}(M)$ s.t $\|g-h\|_\infty\leq \tau$. %For $i_0\in\mathbb{N}$, there exists $(n_{i_0},\ldots, i_0n_{i_0})$ s.t 
then for all $k$ large enough w.r.t $g$, %for all $N\in \{n_k\tau,\ldots, n_k\}$, 
$$\frac{1}{n_k}\sum_{j=0}^{n_k-1}h\circ f^j(y)=\frac{1}{n_k}\sum_{j=0}^{n_k-1}g\circ f^j(y) \pm \tau = \int gd\mu \pm 2\tau = \int hd\mu \pm 3\tau.$$
%Let $n_{i_0+1}>i_0n_{i_0}$ s.t $\forall N\in\{n_{i_0+1},\ldots (i_0+1)n_{i_0+1}\}$, $\frac{1}{N}\sum_{j=0}^{N-1}h\circ f^j(y)= \int hd\mu \pm \frac{3}{i_0+1}$, and construct by recursion $(n_k)_k:=(n_{i_0},\ldots i_0n_{i_0},n_{i_0+1},\ldots,(i_0+1)n_{i_0+1},\ldots  n_{i_0+i},\ldots,(i_0+i)n_{i_0+i},\ldots)$. 
Then %$\overline{d}((n_k)_k)=1$ and 
$\frac{1}{n_k}\sum_{j=0}^{n_k-1}h\circ f^j(y)\xrightarrow[k\to\infty]{}\int hd\mu$.
\end{proof}

\noindent\textbf{Remark:} Note that the weak physicality of $\mu$ does not depend on $\mu$ being hyperbolic, nor on the absolute continuity of the stable foliation.

\begin{cor}
There exists at most one physical measure for $(M,f)$, and if it exists it must be $\mu$.	
\end{cor}

\section{The non-ergodic case}\label{NonErgoCase}

\subsection{Entropy formula via entropy shadowing}

Up until now we treated the case where $\mu$ is ergodic. This simplification serves as a didactic tool to make the proof more intuitive and easy to follow; The proof in the non-ergodic case is more complicated, since we wish to eventually prove that $\int h_{\mu_x}(f)d\mu(x)=\int\chi^u(x)d\mu(x)$ (where $\mu=\int \mu_x d\mu(x)$ is the ergodic decomposition of $\mu$ and $\chi^u(x):=\sum\chi^+(x)$), however it may be that neither of these functions are constant $\mu$-a.e; And so in particular we can neither control $\#\vee_{i=0}^\frac{1}{\epsilon} f^{in\epsilon}[\mathcal{A}_\epsilon^{(n)}]$ by $e^{nh_\mu(f)}=e^{n\int h_{\mu_x}(f)d\mu(x)}$, nor is the volume of each element in $\vee_{i=0}^\frac{1}{\epsilon} f^{in\epsilon}[\mathcal{A}_\epsilon^{(n)}]$ controlled by $e^{-n\int \chi^ud\mu}$. While $\chi^u(\cdot)$ is continuous on Pesin blocks, $h_{\mu_x}$ is merely measurable.

We treat this added difficulty by restricting to Lusin sets, and use a sort of ``entropy shadowing" property, where if a trajectory remains close to different points with good local entropy estimates, then all said shadowed points must have similar local entropy.

\begin{theorem}\label{nonErgSRB}
	$h_\mu(f)=\int \chi^u(x)d\mu(x)$.
\end{theorem}
\begin{proof}
	Let $\mu=\int \mu_xd\mu(x)$ be the ergodic decomposition of $\mu$. Let $\epsilon>0$, and let $E_{j_h,j_{\chi_1},\ldots ,j_{\chi_d}}:=\{x: h_{\mu_x}(f)\in [j_h\cdot \epsilon^5-\frac{\epsilon^5}{2}, j_h\cdot \epsilon^5+\frac{\epsilon^5}{2}), \chi_i(x)\in (j_{\chi_i}\cdot \epsilon^5-\frac{\epsilon^5}{2}, j_{\chi_i}\cdot \epsilon^5+\frac{\epsilon^5}{2}], i\leq d\}$, $\underline j\in\{0,\ldots,\frac{2d\log M_f}{\epsilon^5}\}^{d+1}$. Assume w.l.o.g that $\mu(E_{\underline j})>0$ for all $\underline j$, and let $\rho_\epsilon:=\epsilon\cdot(\min_{\underline j}\{\mu(E_{\underline j})\})^4>0$.
	
For each $\underline j$, we define the set $E_{\rho_\epsilon}^{\underline j}$ as in \textsection \ref{prel}, for the measure $\mu_{\underline j}:=\frac{1}{\mu(E_{\underline j})}\int_{E_{\underline j}}\mu_xd\mu(x)$. Then it follows that $\mu(\bigcupdot_{\underline j}E^{\underline j}_\epsilon)\geq e^{-\rho_\epsilon^3}$.

Let $L_\epsilon$ be a Lusin set for the function $x\mapsto h_{\mu_x}(f)$ s.t $\forall \underline j$, $\mu_{\underline j}(L^{\underline j}_\epsilon)\geq e^{-2\rho_\epsilon^3}$, where $L^{\underline j}_\epsilon :=L_\epsilon\cap E_{\underline j} $. Since the Lusin theorem tells us that $L_\epsilon$ can be chosen to be closed, there exists $0<r_\epsilon:=\frac{1}{2}\sup\{r>0:\forall x,y\in L_\epsilon, d(x,y)\leq r\Rightarrow |h_{\mu_x}(f)-h_{\mu_y}(f) |\leq \epsilon^5\}$. Similarly, assume that $L_\epsilon$ is a Lusin set for the function $x\mapsto \underline\chi(x)$, with the same estimates w.r.t the $|\cdot|_\infty$-norm.

Finally, given $n\in\mathbb{N}$, set $G_\epsilon^{\underline j, n}:= L^{\underline j}_\epsilon \cap f^{n\epsilon}[L^{\underline j}_\epsilon]$, and so $\mu(\bigcupdot_{\underline j}G_\epsilon^{\underline j,n})\geq e^{-\rho_\epsilon^2}$.

Cover each $G_\epsilon^{\underline j,n}$ with $\widetilde{\mathcal{A}}_\epsilon^{\underline j,n}$- a cover by exponential Bowen balls  $B(\cdot,-n\epsilon,e^{-2\delta n})$, as in \textsection \ref{prel}. %Hence $\#\mathcal{A}_\epsilon^{\underline j,n}\geq e^{n\epsilon(h_{\mu_j}(f)-2\epsilon^2)}$. In addition, let $\mathcal{D}_\epsilon^{\underline j,n}$ be a cover of $G_\epsilon^{\underline j,n} $by exponential Bowen balls  $B(\cdot,-n\epsilon^\frac{1}{3},e^{-\delta n})$, as in \textsection \ref{prel} as well. 
Hence $\#\mathcal{A}_\epsilon^{\underline j,n}= e^{n\epsilon(h_{\mu_j}(f)\pm2\epsilon^2)}$% and $\#\mathcal{D}_\epsilon^{\underline j,n}\geq e^{n\epsilon^\frac{1}{3}(h_{\mu_j}(f)-2\epsilon^2)}$
, where $\mathcal{A}_\epsilon^{\underline j,n}:=\{B(x,n\epsilon,e^{-\delta n}): B(x,n\epsilon,e^{-2\delta n})\in \widetilde{\mathcal{A}}_\epsilon^{\underline j,n}\}$, similarly to \textsection \ref{prel}.

Set $\mathcal{S}_n:=\bigcap_{i=0}^{\frac{1-2\sqrt{\epsilon}}{\epsilon}}f^{in\epsilon}[\bigcup_{\underline j}\bigcup\mathcal{A}_\epsilon^{\underline j,n}]$. As in Theorem \ref{firstTheorem}, $\frac{1}{\sqrt\epsilon n}\sum_{k=(1-\sqrt\epsilon)n}^{n-1}m\circ f^{-k}(\mathcal{S}_n)\geq e^{-\rho_\epsilon^\frac{3}{4}}$ (for all $\epsilon>0$ sufficiently small).
	
Then, for any $\underline j$, as in Lemma \ref{firstLemma}, for all $n$ large enough (s.t $\epsilon=\frac{1}{N^6}$ and $N^6|n$), $\frac{1}{\sqrt\epsilon n}\sum_{k=(1-\sqrt\epsilon)n}^{n-1}m\circ f^{-k}(\bigcup\mathcal{A}_\epsilon^{\underline j,n})\geq \frac{1}{2}\mu(G_\epsilon^{\underline j,n})\geq \frac{1}{2}\mu(E_{\underline j})e^{-\rho_\epsilon^2}\gg 2\rho_\epsilon^\frac{3}{4}$, whence 
\begin{equation}\label{NAIC}
	\frac{1}{\sqrt\epsilon n}\sum_{k=(1-\sqrt\epsilon)n}^{n-1}m\circ f^{-k}(\bigcup\mathcal{A}_\epsilon^{\underline j,n}\cap \mathcal{S}_n)\geq \frac{1}{5}\mu(E_{\underline j})\geq \rho_\epsilon.
\end{equation}	
	
Notice, given $B\in \vee_{i=0}^{\frac{1-2\sqrt\epsilon}{\epsilon}}f^{in\epsilon}[\bigcup_{\underline j}\mathcal{A}_\epsilon^{\underline j,n}]$ s.t $B=\bigcap_{i=0}^{\frac{1-2\sqrt{\epsilon}}{\epsilon}}f^{in\epsilon}[B_i]$ with $B_i\in \mathcal{A}_\epsilon^{\underline j^i,n}$, and given $D\in\mathcal{A}_\epsilon^{\underline j,n}$ s.t $D\cap B\neq \varnothing$, we have 
\begin{equation}\label{NAIC3}
|h_{\mu_{\underline j}}(f)-h_{\mu_{\underline j^i}}(f) |\leq \epsilon^3\text{ and }|\int\chi^ud\mu_{\underline j}-\int\chi^ud\mu_{\underline j^i} |\leq \epsilon^3\text{ for all }i\leq\frac{1-2\sqrt{\epsilon}}{\epsilon}
\end{equation}
(as long as $n$ is large enough so $2e^{-\delta n}\leq r_\epsilon$, since $h_{\mu_{\cdot}}(f)= h_{\mu_{f^{n\epsilon}(\cdot)}}(f)$ and $\underline\chi(\cdot)= \underline\chi(f^{n\epsilon}(\cdot))$).

Write $\hat{\mathcal{A}}_\epsilon^{\underline j,n}:=\{D\in \mathcal{A}_\epsilon^{\underline j,n}: \frac{1}{\sqrt\epsilon n}\sum_{k=(1-\sqrt\epsilon)n}^{n-1}m\circ f^{-k}(D)\leq e^{\epsilon^\frac{3}{2}n} \frac{1}{\sqrt\epsilon n}\sum_{k=(1-\sqrt\epsilon)n}^{n-1}m\circ f^{-k}(D\cap \mathcal{S}_n)\}$ and $\check{\mathcal{A}}_\epsilon^{\underline j,n}:= \mathcal{A}_\epsilon^{\underline j,n} \setminus \hat{\mathcal{A}}_\epsilon^{\underline j,n}$. Then $\# \hat{\mathcal{A}}_\epsilon^{\underline j,n} \geq e^{n\epsilon%^\frac{1}{3}
(h_{\mu_j}(f)-2\epsilon^2)} e^{-\epsilon^\frac{3}{2}n}$; otherwise
\begin{align*}
	0<\rho_\epsilon\leq &\frac{1}{\sqrt\epsilon n}\sum_{k=(1-\sqrt\epsilon)n}^{n-1}m\circ f^{-k}(\bigcup\mathcal{A}_\epsilon^{\underline j,n}\cap \mathcal{S}_n)\\
	\leq &\#\check{\mathcal{A}}_\epsilon^{\underline j,n}\cdot e^{-n\epsilon%^\frac{1}{3}
	(h_{\mu_{\underline j}}(f)-\epsilon^2)}e^{-n\epsilon^\frac{3}{2}}+e^{n\epsilon%^\frac{1}{3}
	(h_{\mu_{\underline j}}(f)-\epsilon^2)-n\epsilon^\frac{1}{3}}\cdot e^{-n\epsilon^\frac{1}{3}(h_{\mu_{\underline j}}(f)-\epsilon^2)}\leq 2e^{-\frac{1}{2}n\epsilon^\frac{3}{2}}\xrightarrow[n\to\infty]{}0,
\end{align*}
a contradiction!

\medskip
Now, recall that $\mathcal{A}_\epsilon^{\underline j,n}$ is a cover of multiplicity bounded by $e^{2\epsilon^2 n}$, and hence so is $\hat{\mathcal{A}}_\epsilon^{\underline j,n}$. As in \cite[Lemma~2.2]{NLE}, there exists a pair-wise disjoint sub-cover $\overline{\mathcal{A}}_\epsilon^{\underline j,n}\subseteq \hat{\mathcal{A}}_\epsilon^{\underline j,n} $ s.t $\#\overline{\mathcal{A}}_\epsilon^{\underline j,n}\geq \#\hat{\mathcal{A}}_\epsilon^{\underline j,n} e^{-2\epsilon^2n}$. Finally, notice that $\vee_{i=0}^{\frac{1-2\sqrt\epsilon}{\epsilon}}f^{in\epsilon}[\vee_{\underline j'}\mathcal{A}_\epsilon^{\underline j',n}]$ refines $\overline{\mathcal{A}}_\epsilon^{\underline j,n}$. Therefore it follows that for any $\underline j$, there exists $D_{\underline j}\in \overline{\mathcal{A}}_\epsilon^{\underline j,n}$ s.t 
\begin{align*}
\#\{B\in \vee_{i=0}^{\frac{1-2\sqrt\epsilon}{\epsilon}}f^{in\epsilon}[\vee_{\underline j'}
\mathcal{A}_\epsilon^{\underline j',n}]: B\cap D_{\underline j} \neq \varnothing\}\leq &\frac{\#\{B\in \vee_{i=0}^{\frac{1-2\sqrt\epsilon}{\epsilon}}f^{in\epsilon}[\vee_{\underline j'}
\mathcal{A}_\epsilon^{\underline j',n}]:\exists D'\in \overline{\mathcal{A}}_\epsilon^{\underline j,n}\text{ s.t }B\cap D'\neq \varnothing\}}{\#\overline{\mathcal{A}}_\epsilon^{\underline j,n}}\\
\leq &\frac{C_\epsilon\cdot e^{n(1-2\sqrt\epsilon)(h_{\mu_{\underline j}(f)}+\epsilon^3)}}{e^{n\epsilon%^\frac{1}{3}
 (h_{\mu_{\underline j}(f)}-2\epsilon^2) -n\epsilon^\frac{3}{2} -n\epsilon^2}},\end{align*}
where $C_\epsilon:=(\frac{2d\log M_f}{\epsilon^5})^\frac{d+1}{\epsilon}$. For all $n $ large enough s.t $e^{-n\epsilon^3}\leq \frac{1}{C_\epsilon}$, we have 
\begin{equation}\label{NAIC2}
	\#\{B\in \vee_{i=0}^{\frac{1-2\sqrt\epsilon}{\epsilon}}f^{in\epsilon}[\vee_{\underline j'}
	\mathcal{A}_\epsilon^{\underline j',n}]: B\cap D_{\underline j} \neq \varnothing\}\leq e^{n%(1-\epsilon^\frac{1}{3})
	h_{\mu_{\underline j}}(f)+2\epsilon^\frac{3}{2}n}.
\end{equation}
Therefore, in total,
\begin{align*}
	e^{-n\epsilon%^\frac{1}{3}
	(h_{\mu_{\underline j}}(f)-\epsilon^2)}\leq  \mu(D_{\underline j})\leq &2 \frac{1}{\sqrt\epsilon n}\sum_{k=(1-\sqrt\epsilon)n}^{n-1}m\circ f^{-k}(D_{\underline j})\leq 2 e^{n\epsilon^\frac{3}{2}}\frac{1}{\sqrt\epsilon n}\sum_{k=(1-\sqrt\epsilon)n}^{n-1}m\circ f^{-k}(D_{\underline j}\cap\mathcal{S}_n)\\
	\leq &  2 e^{n\epsilon^\frac{3}{2}} \cdot e^{n%(1-\epsilon^\frac{1}{3})
	h_{\mu_{\underline j}}(f)+2\epsilon^\frac{3}{2}n}\cdot \max_{B\cap D_{\underline j}\neq\varnothing} \frac{1}{\sqrt\epsilon n}\sum_{k=(1-\sqrt\epsilon)n}^{n-1}m\circ f^{-k}(B)\\
	\leq &2 e^{n\epsilon^\frac{3}{2}} \cdot e^{n%(1-\epsilon^\frac{1}{3})
	h_{\mu_{\underline j}}(f)+2\epsilon^\frac{3}{2}n}\cdot \max_{B\cap D_{\underline j}\neq\varnothing} \frac{1}{\sqrt\epsilon n}\sum_{k=(1-\sqrt\epsilon)n}^{n-1}m\circ f^{-(k-(1-2\sqrt\epsilon)n)}(f^{-(1-2\sqrt\epsilon)n}[B])\\
\leq &2 e^{n\epsilon^\frac{3}{2}} \cdot e^{n%(1-\epsilon^\frac{1}{3})
h_{\mu_{\underline j}}(f)+2\epsilon^\frac{3}{2}n}\cdot \max_{B\cap D_{\underline j}\neq\varnothing} e^{2\sqrt{\epsilon} d\log M_f n}m(f^{-(1-2\sqrt\epsilon)n}[B])\\
\leq &2 e^{n\epsilon^\frac{3}{2}} \cdot e^{n%(1-\epsilon^\frac{1}{3})
h_{\mu_{\underline j}}(f)+2\epsilon^\frac{3}{2}n}\cdot  e^{2\sqrt{\epsilon} d\log M_f n}e^{-n(1-2\sqrt\epsilon)(\chi^u(\mu_{\underline j}) -\epsilon^2)} e^{2\epsilon^\frac{1}{3}n}.
\end{align*}	
	Where the last inequality is by \eqref{NAIC3} similarly to Lemma  \ref{tubeVol}. Then,
	\begin{align*}
		e^{-nh_{\mu_{\underline j}}(f)}\leq e^{-n\chi^u(\mu_{\underline j})}e^{10d\log M_f\epsilon^\frac{1}{3} n},
	\end{align*}
		and since $\epsilon>0$ is allowed to be arbitrarily small, for $\mu$-a.e $x$ $h_{\mu_x}(f)\geq \chi^u(x)$, and we are done.
\end{proof}

\section{Volume is almost exponentially mixing}\label{almostExpMix}

\subsection{Exponential decay of correlations implies exponential convergence}

While checking the condition at \eqref{expmix} may seem difficult since one has to a-priori know the measure $\mu$, here we present a condition which implies \eqref{expmix} without comparing to an explicit $f$-invariant measure.

\begin{prop}\label{expCorToExpMix}
	Assume that there exist $C,\gamma,\alpha>0$ s.t for all $g,h\in \mathrm{H\ddot{o}l}_{\alpha}(M)$ s.t $\int h dm=0$,
	$$\left|\int g\circ f^n\cdot hdm\right|\leq C\|g\|_\alpha\|h\|_\alpha e^{-\gamma n}.$$
	Then there exists an $f$-invariant Borel probability $\mu$ s.t $m\circ f^{-n}\xrightarrow[n\to\infty]{}\mu$ exponentially fast, and moreover \eqref{trueExpMix} holds.
\end{prop}
\begin{proof}
	We start by proving that the exponential convergence property applies to the Jacobian of $f$ (recall that $d_\cdot f$, $d_\cdot f^{-1}$ are $\beta$-H\"older continuous):
	
\medskip
\noindent\textbf{Step 1:} $\exists C''\geq C$, $\gamma'\in (0,\gamma)$ s.t $\forall g,h\in \mathrm{H\ddot{o}l}_{\beta}(M)$ with $\int h dm=0$, $\left|\int g\circ f^n\cdot hdm\right|\leq C''\|g\|_\beta\|h\|_\beta e^{-\gamma' n}$.

\medskip
\noindent\textbf{Proof:} Let $\{\psi_i\}_{i=1}^N$ be a partition of unity of $M$, s.t for all $i\leq N$, $0\leq \psi_i\leq 1$ is a $C^\infty$ function which is supported on a ball $U_i$ with a $C^\infty$ diffeomorphism $\Theta_i: B_{\mathrm{R}^d}(0,1)\to U_i$. Let $g$ be a $\beta$-H\"older function (w.l.o.g $\|g\|_\infty\leq 1$), and assume that $g$ is supported on some set $U_i$. 

Let $\eta>0$ to be determined later, and let $K_n:B_{\mathrm{R}^d}(0,e^{-\eta n})\to\mathbb{R}$ be the cone function kernel: $K_n(t)=(e^{-\eta n}-|t|)\cdot \frac{H_n}{e^{-\eta n}}$ where $H_n=\frac{C_d}{e^{-\eta nd}}$ satisfies $\int K_n(t) dt=1$. Extend $K_n$ naturally to $\mathbb{R}^d$ by taking the value $0$ outside $B_{\mathrm{R}^d}(0,e^{-\eta n})$. Set $g_n:=(g\circ \Theta_i* K_n)\circ \Theta_i^{-1}$. Then one can check that the following properties hold:
\begin{enumerate}
	\item $ \|K_n\|_{\alpha}\leq\|K_n\|_{\mathrm{Lip}}=\frac{H_n}{e^{-\eta n}}=C_d\cdot e^{\eta n (d+1)}$ where $C_d$ is a constant depending only on $d$,
	\item $\|g_n\|_\alpha\leq \Vol(B_{\mathrm{R}^d}(0,1))\cdot \mathrm{Lip}(\Theta_i^{-1})\cdot \|K_n\|_\alpha$,
	\item $|g_n(x)-g(x)|\leq \mathrm{Lip}(\Theta_i)\cdot \|g\|_\beta\cdot e^{-\eta n\beta}$.
\end{enumerate}

Notice, given $g,h\in \mathrm{H\ddot{o}l}_{\beta}(M)$ with $\int h dm=0$ (w.l.o.g $\|g\|_\infty,\|h\|_\infty\leq 1$), we can write $g= \sum_{i=1}^N \psi_i\cdot g$, $h= \sum_{i=1}^N \psi_i\cdot h$, and $\sum_{i=1}^N \int h\cdot psi_i dm=0$, where $\psi_i g$ and $\psi_i h$ are $\beta$-H\"older for all $i\leq N$. Therefore, $$\int g\circ f hdm=\sum_{i=1}^N\sum_{j=1}^N\int (\psi_i g)\circ f^n (\psi_j h-\int\psi_j hdm)dm.$$
Set $g^i:=\psi_i\cdot g$ and $h^j:=\psi_j\cdot h-\int \psi_j\cdot hdm $, then 
\begin{align*}
\int g^i\circ f^n h^jdm=& \int (g^i-g^i_n)\circ f^n h^jdm + \int g^i_n\circ f^n (h^j-h^j_n)dm+ \int g^i_n\circ f^n  \cdot (\int h^j_n dm)dm \\
+&	 \int g^i_n\circ f^n \cdot (h^j_n-\int h^j_n dm)dm.
\end{align*}
On the r.h.s, all three first summands are exponentially small by item (3). The last summand on the r.h.s is bounded by $Ce^{-\gamma n}\|g^i_n\|_\alpha \|h^j_n\|_\alpha\leq Ce^{-\gamma n}\cdot C_d^2 \Vol(B_{\mathrm{R}^d}(0,1))^2\mathrm{Lip}(\Theta_i^{-1}) \mathrm{Lip}(\Theta_j^{-1}) e^{2\eta n (d+1) }\leq C'_d C_0 e^{-\frac{\gamma}{2}n}$ whenever $\eta:=\frac{\gamma}{4(d+1)}$ and $C_0:=\max_{i\leq N}\{\mathrm{Lip}(\Theta_i), \mathrm{Lip}(\Theta_i^{-1})\}$. Therefore for $\gamma':=\eta \beta\in (0,\frac{\gamma}{2})$,
\begin{align*}
	\left|\int g\circ f^n hdm\right|\leq& N^2\cdot \left(C_0\|g\|_\beta e^{-\eta\beta n}+(1+C_0\|g\|_\beta e^{-\eta n\beta}) C_0\|h\|_\beta e^{-\eta n\beta}+C_0\|h\|_\beta e^{-\eta n\beta} + C'_d C_0 e^{-\frac{\gamma}{2}n}\right)\\
	\leq & C''\|g\|_\beta\|h\|_\beta e^{-\gamma' n}.
\end{align*}

\medskip
\noindent\textbf{Step 2:} There exist a constant $C_f''>0$ and an $f$-invariant Borel probability $\mu$ s.t $\forall g,h\in \mathrm{H\ddot{o}l}_{\beta}(M)$,
$$\left|\int g\circ f^n hdm-\int gd\mu\int hdm\right|\leq C_f''e^{-\gamma'n}\|g\|_\beta \|h\|_\beta.$$

\medskip
\noindent\textbf{Proof:} Fix $h\in \mathrm{H\ddot{o}l}_{\beta}(M)$ with $h\geq0$ and $h\not\equiv 0$, and let $g\in \mathrm{H\ddot{o}l}_{\beta}(M)$. Define the sequence $a_n^h(g):= \int g\circ f^n hdm$, and notice 
$$\left| a_{n+1}^h(g)-a_n^h(g)\right|=\left|  \int g\circ f^n \Jac(f^{-1})h\circ f^{-1}dm- \int g\circ f^n hdm\right|=\left| \int g\circ f^n H dm\right|,$$
where $H:= \Jac(f^{-1})h\circ f^{-1} -h\in \mathrm{H\ddot{o}l}_{\beta}(M) $ with $\int Hdm=0$. Therefore, $| a_{n+1}^h(g)-a_n^h(g)| \leq C_f\|h\|_\beta\|g\|_\beta e^{-\gamma' n}$, for a constant $C_f$ depending on $f$. Thus the sequence $\{a_n^h(g)\}_{n\geq 0}$ is a Cauchy sequence and has a limit defined $\int h dm \cdot \mu_h(g)$. Notice:
\begin{enumerate}
	\item $\mu_h(a_1g_1+a_2g_2)= a_1\mu_h(g_1)+a_2\mu_h(g_2)$,
	\item $\mu_h(1)=1$,
	\item $g\geq 0\Rightarrow \mu_h(g)\geq 0$,
	\item $\mu_h(g)\leq \|g\|_\infty$.
\end{enumerate}
Therefore by the Riesz representation theorem, $\mu_h(\cdot)$ defined a Borel probability measure on $M$. Moreover, it is easy to check that $\mu_h(g)= \mu_h(g\circ f)$ from definition, therefore $\mu_h$ is $f$-invariant.

Given $h_1, h_2\in \mathrm{H\ddot{o}l}_{\beta}(M)$ with $h_1,h_2\geq0$ and $h_1,h_2\not\equiv 0$, then 
$$\left|\frac{1}{\int h_1 dm}a_n^{h_1}(g)-\frac{1}{\int h_2dm}a_n^{h_2}(g)\right|= \left|\int g\circ f^n \cdot \left(\frac{h_1}{\int h_1dm}-\frac{h_2}{\int h_2dm}\right)\right| \xrightarrow[n\to\infty]{}0.$$
Therefore $\mu_h(g)$ is independent of the choice of $h$, and can be denoted by $\mu(g)$. Given $h \in \mathrm{H\ddot{o}l}_{\beta}(M)$ with $\|h\|_\infty\leq1$ and $h\not\equiv 0$, we can write $h=h^+-h^-$, where $0\leq h^+,h^-\leq 1$ and $\|h^+\|_\beta, \|h^-\|_\beta\leq \|h\|_\beta$. Assume w.l.o.g $h^+,h^-\not\equiv 0$. Then for $\sigma\in\{-,+\}$,
\begin{align*}
	\left|a_n^{h^\sigma}(g)-\int h^\sigma dm\cdot \mu(g)\right|\leq C_f'\|h\|_\beta\|g\|_\beta e^{-\gamma'n}\Rightarrow \left|a_n^{h}(g)-\int hdm\cdot \mu(g)\right|\leq 2C_f'\|h\|_\beta\|g\|_\beta e^{-\gamma'n}.
\end{align*}
In particular, for $h\equiv 1$, we get $m\circ f^{-n}\xrightarrow[\exp]{}\mu$.
\end{proof}

\subsection{Positive entropy, ergodicity, and uniqueness}\label{newSec4}

In this section we assume the following strong notion of exponential convergence:
\begin{equation}\label{NAIC4}
	\exists C>0,\alpha\in(0,1],\gamma>0:\forall g,h\in \mathrm{H\ddot{o}l}_\alpha(M), \left|\int g\circ f^n \cdot hdm-\mu(g)\cdot m(h)\right|\leq C\cdot \|g\|_\alpha\cdot \|h\|_\alpha\cdot e^{-\gamma n}.
\end{equation}

Recall Proposition \ref{expCorToExpMix}, where we show that \eqref{NAIC4} holds whenever the volume is almost exponentially mixing (recall \eqref{expDecay}). The condition of almost exponential mixing of the volume is inspired by the setup studied by Dolgopyat, Kangowski, and Rodriguez-Hertz in \cite{ExpMixBern} (however notice that the volume need \textbf{not} be invariant in our setup).

We continue to show that under the assumption of \eqref{NAIC4}, unless $\mu$ is a Dirac delta measure at a fixed point (a necessary condition, see the remark following Theorem \ref{posExps}), indeed $\mu$ must be an ergodic %mixing 
SRB measure with at least one positive Lyapunov exponent almost everywhere, and it is the unique SRB measure of $(M,f)$. %This allows us to extend the arguments of \cite{ExpMixBern} and show that $\mu$ is a Bernoulli measure. 
A nice corollary of our proof is that every $f$-invariant Borel probability measure on $M$ has a uniform bound form below on its maximal Lyapunov exponent in terms of the rate of mixing, aside at most for $\mu$ in the case where $\mu$ is a Dirac delta measure.

\begin{theorem}[Positive exponents]\label{posExps}
The following dichotomy holds:
\begin{enumerate}
	\item for every ergodic $f$-invariant Borel probability $\nu$, $\max_i\chi^+_i(x)>\frac{\gamma}{2d}$ $\nu$-a.e,
	\item $\mu$ is a Dirac delta measure at a fixed point with $\chi^u(\mu)=0$, and every other ergodic $f$-invariant Borel probability $\nu$ has $\max_i \chi^+_i(x)> \frac{\gamma}{2d}$ $\nu$-a.e.
\end{enumerate}
\end{theorem}
In particular, if $\mu$ is not a Dirac delta measure at a fixed point, then $\max_i \chi^+_i(x)\geq \frac{\gamma}{2d} $ $\mu$-a.e.
\begin{proof}
Let $\nu$ be an ergodic $f$-invariant Borel probability s.t $\max_i\chi_i^+(x)\leq \frac{\gamma}{2d}-8(d+1)\epsilon$ $\nu$-a.e, where w.l.o.g $0<\epsilon\ll \frac{\gamma}{2d}$. Let $x$ be a $\nu$-typical point. Let $n$ large s.t $f^i[B(x,n,e^{-\epsilon n})]$ is contained in the Pesin chart of $f^i(x)$ for all $0\leq i\leq n$. Let $g_x$ be a Lipschitz function s.t $g_x|_{B(x,n,e^{-2\epsilon n})}=1$, $g_x|_{B(x,n,e^{-\epsilon n})^c}=0$, and $\mathrm{Lip}(g_x)\leq e^{(\frac{\gamma}{2d}-3(d+1)\epsilon)n }$. 

Let $p$ and $q$ s.t $\mu(B(p, e^{-\epsilon n})), \mu(B(q, e^{-\epsilon n}))\geq e^{-n(d+1)\epsilon}$ for all $n$ large enough, and let $g_t|_{B(t,e^{-2\epsilon n})}=1$, $g_t|_{B(t,e^{-\epsilon n})^c}=0$, and $\mathrm{Lip}(g_t)\leq e^{2\epsilon n}$, $t\in\{p,q\}$. Then by \eqref{NAIC4}, for $t\in \{p,q\}$ and all $n$ large enough,
\begin{align*}
\int g_t\circ f^ng_xdm=&\mu(g_t)m(g_x)\pm 4Ce^{-\gamma n}e^{2\epsilon n}e^{(\frac{\gamma}{2d}-3(d+1)\epsilon)n }\\
=&e^{\pm \epsilon} \mu(g_t)m(g_x)>0 \text{ }(\because \mu(g_t)m(g_x)\geq e^{-n(d+1)2\epsilon-(\frac{\gamma}{2d}+\epsilon)dn}).
\end{align*}
Thus in particular $B(p,e^{-\epsilon n})\cap B(f^n(x),e^{-\epsilon n})\neq \varnothing$ and $B(q,e^{-\epsilon n})\cap B(f^n(x),e^{-\epsilon n})\neq \varnothing$, and so $d(p,q)\leq 4e^{-\epsilon n}\xrightarrow[n\to\infty]{}0$. Therefore $\mu=\delta_p=\delta_q$.

Moreover, if we assume further that $x$ is $\nu$-generic, for any $h\in \mathrm{Lip}_+(M)$, 
\begin{align*}
m(g_x)(h\circ f^n(x)\pm \mathrm{Lip}(h)e^{-\epsilon n})=&\int h\circ f^ng_xdm=\mu(h)m(g_x)\pm 4Ce^{-\gamma n}e^{\epsilon n}e^{(\frac{\gamma}{2d}-3(d+1)\epsilon) n}\\
= &m(g_x)(\mu(h)\pm 4Ce^{-\frac{\gamma}{2} n}e^{\epsilon n}e^{(\frac{\gamma}{2d}-3(d+1)\epsilon) n}e^{\epsilon d n}).
\end{align*}
Averaging over $n=N,\ldots ,2N$, for $N\in \mathbb{N}$ large,
\begin{align*}
\nu(h)\xleftarrow[\infty\leftarrow N]{}\frac{1}{N}\sum_{n=N}^{2N-1}h\circ f^n(x)\pm \mathrm{Lip}(h)e^{-\epsilon N}=\mu(h)\pm 4Ce^{-\frac{\gamma}{2} N}e^{\epsilon N}e^{(\frac{\gamma}{2d}-3(d+1)\epsilon) N}e^{\epsilon d N}\xrightarrow[N\to\infty]{}\mu(h).
\end{align*}
Therefore $\mu(h)=\nu(h)$, and so by the Riesz representation theorem, $\nu=\mu$ (since $\overline{\mathrm{Lip}_+(M)}^{C(M)}=C_+(M)$).
\end{proof}
\noindent\textbf{Remark:} The assumption that $\mu$ is not a Dirac delta measure is necessary, as can be seen by the north-pole south-pole example: Let $\mathbb{S}^1$ be the unit circle, let $f\in\mathrm{Diff}^{1+\beta}(\mathbb{S}^1)$, and let two fixed points, $N\in \mathbb{S}^1 $ with a derivative bigger than $1$, and $S\in \mathbb{S}^1 $ with a derivative smaller than $1$. One can check that in this case $|\int_{\mathbb{S}^1} g\circ f^n h dm-g(S)m(h)|$, $g,h\in \mathrm{Lip}(\mathbb{S}^1)$, is exponentially small as in \eqref{NAIC4}, by using a partition of unity which separates $N$ and $S$. This example extends to a closed surface using a D-A system with two repellers, and a fixed attracting point.

\medskip
By Theorem \ref{nonErgSRB} and Theorem \ref{posExps}, it follows from \cite{LedrappierYoungI} that $\mu$ has absolutely continuous conditionals on unstable leaves a.e. The following theorem is a corollary of this fact together with \eqref{NAIC4}. The proof uses absolutely continuous fake $cs$-foliations in exponentially small charts, constructed in \cite{ExpMixBern}. These foliations are used to treat the trajectory of an exponentially small ball as the trajectory of single unstable leaf, where the conditional measure of $\mu$ is equivalent to the induced Riemannian volume, which lets us compare the two measures.
\begin{prop}\label{UMix}
Assume that there exists an ergodic SRB measure $\nu$ with $h_\nu(f)>0$. For every $\epsilon\in (0,\frac{\gamma}{4\log M_f})$ there is a set $G_\epsilon$ with $\nu(G_\epsilon)\geq e^{-\epsilon}$ s.t $\forall x%,y
\in G_\epsilon$, $\forall \delta\in(0,\epsilon)$, $\forall n\geq n_{\epsilon,\delta}$, $\forall g\in \mathrm{H\ddot{o}l}_\alpha^+(M)$, \begin{comment}$$\int %\mathbb{1}_{B(x,n\epsilon,e^{-n\delta})}
g\circ f^n \mathbb{1}_{B^u(x,%n\epsilon,
e^{-n\delta})} d\nu_{\xi^u(x)}=(\mu(%B(x,n\epsilon,e^{-n\delta})
g)\cdot \nu_{\xi^u(x)}(B^u(x%,n\epsilon
,e^{-n\delta}) )\pm e^{-n\frac{\delta\alpha}{3} }\|g\|_{\alpha\mathrm{-H\ddot{o}l}}\cdot \nu_{\xi^u(x)}(B^u(x%,n\epsilon
,e^{-n\delta}) ))e^{\pm\delta \|g\|_\infty},$$
\end{comment}
\begin{equation}\label{BeforeLasteq}
\mu(g)\geq  \frac{e^{-7\delta^2 d}}{\nu_{\xi^u(x)}(B^u(x,e^{-\delta n})\cap K_\epsilon)
\mathcal{W}^\mathrm{cs}_n(x))}\int\limits_{B^u(x,e^{-\delta n})\cap K_\epsilon} (g\circ f^n-\|g\|_{\alpha\mathrm{-H\ddot{o}l}}e^{-\frac{\delta}{2}n\alpha}) d\nu_{\xi^u(x)}-C\|g\| _{\alpha\mathrm{-H\ddot{o}l}} e^{-(\gamma-2\delta)n} \end{equation}
where $\xi^u$ is a measurable partition subordinated to the unstable foliation of $\nu$, $\nu_{\xi^u(\cdot)}$ are the respective conditional measures, and $K_\epsilon$ is a Pesin block with $\nu(K_\epsilon)\geq e^{-\epsilon^2}$.
\end{prop}
For the definition of a measurable partition subordinated to the unstable foliation of $\nu$, see \cite{LedrappierYoungI}, and the respective conditional measures exists $\nu$-a.e by the Rokhlin disintegration theorem.
\begin{proof}\text{ }

\textbf{Step 1:} $\nu_{\xi^u(x)}= C^{\pm1 }\frac{1}{m_{\xi^u(x)}(1)}m_{\xi^u(x)}$ on a large set, where $m_{\xi^u(x)}$ is the induced Riemannian volume on $\xi^u(x)$ and $C>0$ is a constant close to $1$.

\textbf{Proof:} Let $\xi^u$ be a partition measurable partition subordinate to the unstable foliation of $\nu$ s.t $\xi^u(x)\supseteq B^u(x,r_x)$ for $\nu$-a.e $x$ (see \cite{LedrappierYoungI}). Let $\nu=\int \nu_{\xi^u(x)}d\nu(x)$ be the corresponding disintegration of $\nu$ given by the Rokhlin disintegration theorem. By %Theorem \ref{nonErgSRB}, 
Theorem \ref{posExps}, and \cite{LedrappierYoungI}, for $\nu$-a.e $x$, $\nu_{\xi^u(x)}\sim m_{\xi^u(x)}$. 
	
Denote by $\rho_x$ the Radon-Nikodym derivative $\frac{d\nu_{\xi^u(x)}}{dm_{\xi^u(x)}}$. By the construction of $\xi^u$, $\rho_x(x)$ is uniformly bounded a.e, since the elements of $\xi^u$ are contained in local unstable leaves, and moreover $\log \rho_x$ is $\frac{\beta}{3}$-H\"older continuous with a uniform H\"older constant (see \cite{LedrappierYoungI} for more details of this classical result).

 Therefore, given $\epsilon>0$ and a small $\delta\in(0,\epsilon)$, there exists $\ell_\epsilon=\ell_\epsilon(\delta)$ s.t $\nu\left(\Lambda_{\ell_\epsilon}^{(\underline\chi(\nu),\delta^3\tau_{\underline\chi(\nu)})}\right)\geq e^{-\epsilon^2}$. Let %$\ell_\epsilon$ large so $\mu(\{x:\ell_x^\epsilon\leq \ell_\epsilon\})\geq e^{-\epsilon^2}$, let 
 $0<\chi\leq \min\{\chi_i(\nu): \chi_i(\nu)\neq 0,i\leq d\}$ (w.l.o.g $\delta\leq \frac{\chi}{2}$), and set $K_\epsilon:=%\bigcup_{[\ell_x^\epsilon\leq \ell_\epsilon, \chi^+_{\min}(x)\geq \chi]}
  \Lambda_{\ell_\epsilon}^{(\underline\chi(\nu),\delta^3\tau_{\underline\chi(\nu)})}$. For all $x\in K_\epsilon$, the local unstable leaf of $x$ contains a relative open ball of radius at least $\frac{1}{2\ell_\epsilon}$.

\textbf{Step 2:} Absolutely continuous fake $cs$-foliation in $B(x,n\epsilon ,e^{-\delta n})$, for $x\in K_\epsilon$, by \cite[\textsection~5,6]{ExpMixBern}.

\textbf{Proof:} Given $x\in K_\epsilon$, and $n$ large enough so $e^{-\frac{\delta}{3} n}\ll \frac{1}{\ell_\epsilon}$, let $\Vcs_n(x)$ be a ``fake central-stable leaf" at $x$, constructed in \cite[Lemma~2.6]{ExpMixBern}. That is, for every $x\in K_\epsilon$ there exists %$L=L(\Lambda_{\ell_\delta}^{(\underline\chi,\tau)})$, and $H^\mathrm{cs}_n(x)$- 
a local submanifold of $x$ transversal to $\xi^u(x)$, $\Vcs_n(x)$, s.t $\forall 0\leq i\leq n$, 
\begin{enumerate}
	\item $f^i[B_{\Vcs_n(x)}(x,e^{-\delta n})]\subseteq B(f^{i}(x),e^{-\frac{\delta}{2} n})$, 
	\item $f^{i}[\Vcs_n(x)]$ is a graph of a function with a Lipschitz constant smaller or equal to $\frac{2\tau}{\chi}\leq \delta^2$ over $\psi_{f^{i}(x)}[\mathbb{R}^\mathrm{cs}\cap B(0,e^{-\frac{\delta}{2} n})]$ (where $\psi_y$ is the Pesin chart of $y$),
	\item $\{\Vcs(y):y\in B(x,e^{-\delta n})\cap K_\epsilon\}$ is a foliation of $B(x,e^{-\delta n})\cap K_\epsilon $ (\cite[\textsection~5]{ExpMixBern}).
\end{enumerate}

Moreover, by \cite[Proposition~6.4]{ExpMixBern},
\begin{enumerate}
	\item[(4)] For all $n$ large enough (when $\delta>0$ is small enough), the holonomy map $\pi$ along $\{\Vcs(x')\}_{x'\in K_\epsilon\cap \xi^u(x)}$ from $\xi^u(x)\cap K_\epsilon$ to $\xi^u(y)$, $y\in K_\epsilon\cap B(x,e^{-\delta n})$, has a Jacobian $\Jac(\pi)=e^{\pm \delta^2}$.
\end{enumerate}
In fact it follows that $\Jac(\pi|_{B^u(x,e^{-\delta n})})=e^{o(1)}$.

\textbf{Step 3:} For every $x\in K_\epsilon$ and $n$ large enough, and for every $g\in\mathrm{H\ddot{o}l}_\alpha^+(M)$, $\frac{1}{\nu_{\xi^u(x)}(%B^u(x,e^{-\delta n})
\mathcal{W}^\mathrm{cs}_n(x))}\int_{%B^u(x,e^{-\delta n})
\mathcal{W}^\mathrm{cs}_n(x)} g\circ f^nd\nu_{\xi^u(x)}=(\mu(g)\pm e^{-\frac{\gamma}{2}n}\|g\|_{\alpha-\mathrm{H\ddot{o}l}})e^{\pm\delta}$, where $\mathcal{W}^\mathrm{cs}_n(x)%:=\bigcup_{y\in K_\epsilon\cap B(x,e^{-\delta n})\cap \xi^u(x)}V_n^\mathrm{cs}(y)\cap B(x,e^{-\delta n})
$ is a foliation box in the chart of $x$.

\textbf{Proof:} We start by defining $\mathcal{W}^\mathrm{cs}_n(x)$ for $x\in K_\epsilon$. Let $R(x,e^{-\delta n}e^{2\delta^2}):=\psi_x(R(0,e^{-\delta n}e^{2\delta^2}))$, where $\psi_x$ is the Pesin chart of $x$, and $R(\cdot,r)$ is a ball of radius $r$ w.r.t to the metric $|\cdot|':=\max\{|\cdot_\mathrm{u}|_2, |\cdot_\mathrm{cs}|_2\}$, where $\mathrm{u},\mathrm{cs}$ are the corresponding components in the chart of $x$. In particular, given $y\in K_\epsilon\cap B^u(x,e^{-\delta n})$, $B_{V^\mathrm{cs}_n(y)}(y,e^{-\delta n})\subseteq R(x,e^{-\delta n}e^{2\delta^2})$, since $V^\mathrm{cs}_n(y)$ is the graph of a $\delta^2$-Lipschitz function in the chart of $x$. We define $\mathcal{W}^\mathrm{cs}_n(x):=\bigcup_{y\in K_\epsilon\cap B^u(x,e^{-\delta n})} B_{V^\mathrm{cs}_n(y)}(y,e^{-\delta n}) $. 

Let $x$ be a $\nu_{\xi^u(x)}$-density point of $K_\epsilon$ s.t $\frac{\nu_{\xi^u(x)}(K_\epsilon\cap B^u(x,r))}{\nu_{\xi^u(x)}(B^u(x,r))}\geq e^{-\delta^2}$, $\forall r\in(0,2e^{-\delta n})$. %For $x\in K_\epsilon$, let $m'=\frac{1}{(m_{\xi^u(x)}\times m_{\Vcs(x)})(B(x,n\epsilon,e^{-\delta n}))}(m_{\xi^u(x)}\times m_{\Vcs(x)})|_{B(x ,n\epsilon,e^{-\delta n})}$, and notice that for all $n$ large enough, $m'=e^{\pm e^{-\frac{2}{\beta}\delta n}}\frac{1}{m(B(x ,n\epsilon,e^{-\delta n}))}m|_{B(x ,n\epsilon,e^{-\delta n})}$ by the local product structure of the Riemannian volume, and the fact that $\xi^u(x)$ and $\Vcs$ are smooth $C^{1+\frac{2\beta}{3}}$ manifolds.
By the H\"older continuity of $\log \rho_x$, $m_{\xi^u(x)} =(\rho_x(x))^{-1}e^{\pm 2e^{-\frac{\beta}{3}\delta n} }\mu_{\xi^u(x)}$ on $B^u(x,2e^{-\delta n})$. Therefore, $\frac{m_{\xi^u(x)}(K_\epsilon\cap B(x,e^{-\delta n}))}{m_{\xi^u(x)}(B(x,e^{-\delta n}))}\geq e^{-2\delta^2}$ for all $n$ large enough. Finally, let $h$ be a Lipschitz function s.t $h|_{R(x,e^{-\delta n}e^{\delta^2})}=1$, $h|_{R(x, e^{-\delta n} e^{2\delta^2})^c}=0$, $\mathrm{Lip(h)}\leq 2e^{2\delta n }$. In particular, $m(\mathcal{W}^\mathrm{cs}(x))\geq e^{-2\delta^2d}m(h)$.

In addition, by %\cite[Corollary~6.6]{ExpMixBern}
the absolute continuity of the foliation $\mathcal{W}^\mathrm{cs}(x)$ and since the induced leaf volume of each laminate in $\mathcal{W}^\mathrm{cs}(x) $ is comparable up to a $e^{\pm2\delta^2d}$ factor, $\frac{1}{m(%B(x,e^{-\delta n})
\mathcal{W}^\mathrm{cs}_n(x))}m|_{%B(x,e^{-\delta n})
\mathcal{W}^\mathrm{cs}_n(x)}=e^{\pm 3\delta^2d}\frac{1}{m_{\xi^u(x)}(%B(x,e^{-\delta n})
\mathcal{W}^\mathrm{cs}_n(x))}m_{\xi^u(x)}|_{%B(x,e^{-\delta n})
\mathcal{W}^\mathrm{cs}_n(x)}$ for sets saturated by $W^\mathrm{cs}(x)$ %in $B(x,e^{-\delta n})$ 
for all $n$ large enough. Thus by Step 2,
\begin{align}\label{lasteq}
	\frac{1}{m(h)}\int h\cdot g\circ f^ndm=& \frac{m(\mathcal{W}^\mathrm{cs}_n(x))}{m(h)}\cdot \frac{1}{m(\mathcal{W}^\mathrm{cs}_n(x))}\int h\cdot g\circ f^ndm\nonumber\\
	\geq &\frac{m(\mathcal{W}^\mathrm{cs}_n(x))}{m(B(x,e^{-\delta n}))}\cdot \frac{m(B(x,e^{-\delta n}))}{m(h)}\cdot\frac{1}{m(\mathcal{W}^\mathrm{cs}_n(x))}\int \mathbb{1}_{\mathcal{W}^\mathrm{cs}_n(x)}\cdot g\circ f^ndm\nonumber\\
	\geq&  e^{-2\delta^2d}\cdot e^{-2\delta^2 d}\cdot e^{-3\delta^2d} \frac{1}{m_{\xi^u(x)}(\mathcal{W}^\mathrm{cs}_n(x))}\int\limits_{\mathcal{W}^\mathrm{cs}_n(x)} (g\circ f^n-\|g\|_{\alpha\mathrm{-H\ddot{o}l}}e^{-\frac{\delta}{2}n\alpha}) dm_{\xi^u(x)}\nonumber\\
	\geq&  e^{-7\delta^2 d} \frac{1}{\nu_{\xi^u(x)}(\mathcal{W}^\mathrm{cs}_n(x))}\int\limits_{\mathcal{W}^\mathrm{cs}_n(x)} (g\circ f^n-\|g\|_{\alpha\mathrm{-H\ddot{o}l}}e^{-\frac{\delta}{2}n\alpha}) d\nu_{\xi^u(x)}.
%	
%	e^{\pm2\delta^2} \frac{1}{m_{\xi^u(x)}(B(x ,e^{-\delta n}))}\int\limits_{B(x ,e^{-\delta n})} (g\circ f^n\pm\|g\|_{\alpha\mathrm{-H\ddot{o}l}}e^{-\frac{\delta}{2}n\alpha}) dm_{\xi^u(x)}\\
%	=& e^{\pm6\delta^2 } \frac{1}{\nu_{\xi^u(x)}(B(x ,e^{-\delta n}))}\int\limits_{B(x ,e^{-\delta n})\cap K_\epsilon} (g\circ f^n\pm\|g\|_{\alpha\mathrm{-H\ddot{o}l}}e^{-\frac{\delta}{2}n\alpha}) d\nu_{\xi^u(x)}.
\end{align}
By \eqref{NAIC4}, the l.h.s equals to $\mu(g)\pm C\cdot \|g\| _{\alpha\mathrm{-H\ddot{o}l}} e^{-(\gamma-2\delta)n}$, and for all $\delta>0$ small enough, we are done by choosing $G_\epsilon$ to be a density set of $K_\epsilon$ with uniform estimates as in step 3.
\end{proof}

\noindent\textbf{Remark:} An upper bound for \eqref{BeforeLasteq} can be achieved similarly through \eqref{lasteq}, although the error term will not be exponentially small in $n$; however the error term is to the already averaged quantity.

\begin{cor}[Uniqueness]
The system $(M,f)$ admits %at most 
exactly one ergodic SRB measure% with positive entropy
, and %if it exists 
it is the measure $\mu$. In particular, $\mu$ is ergodic.
\end{cor}
\begin{proof}
Let $\nu$ be an ergodic SRB measure (in the sense of the entropy formula), we wish to show that $\nu=\mu$. Therefore assume that $\nu\neq\mu$, and by Theorem \ref{posExps}, $h_\nu(f)=\sum\chi^+(\nu)>0$. Let $g\in \mathrm{Lip}(M)$ s.t $0\leq g\leq 1$. Let $\epsilon>0$, and $x\in G_\epsilon$ which is $\nu$-generic for $g$. Assume further that $x$ is a $\nu_{\xi^u(x)}$-density point of $E_{n'}:=\{x'\in K_\epsilon:\forall n\geq n', \frac{1}{n}\sum_{k=n}^{2n-1}g\circ f^k=e^{\pm \delta}\nu(g)\}$ s.t $\frac{\nu_{\xi^u(x)}(B^u(x,e^{-\delta n})\cap E_{n'})}{\nu_{\xi^u(x)}(B^u(x,e^{-\delta n}))}\geq e^{-\delta}$, for some large $n'$. Then, for all $n$ large enough,
	$\mu(g)\geq e^{-\delta}(e^{-7d\delta^2}\nu(g)- \|g\|_{\mathrm{Lip}}e^{-\frac{n\alpha\delta}{3}})$. If $\nu(g)=0$, then $\mu(g)=0$, otherwise for all $n$ large enough w.r.t $g$, $\mu(g)\geq e^{-8d\delta^2}\nu(g)$, hence $\nu=\mu$. In particular, $\mu\geq e^{-8d\delta^2}\nu$ (by the Riesz representation theorem, and since $\overline{\mathrm{Lip}^+(M)}^{C(M)}=C(M)$). Since $\delta>0$ is arbitrary, $\mu\geq \nu$ for every ergodic SRB measure $\nu$ s.t $h_\nu(f)>0$.
	
Assume that $\mu$ can be written as $\mu=a\mu_1+(1-a)\frac{\mu-\mu_1}{1-a}$ with $a\in (0,1)$ and $\mu_1\perp (\mu-\mu_1)$. If $\mu$ admitted an ergodic component with no positive Lyapunov exponents, then by Theorem \ref{posExps} $\mu$ is a Dirac mass at a fixed point, which contradicts the fact that $\mu\geq \nu$ with $h_\nu(f)>0$. Therefore $\mu_1$ admits a positive Lyapunov exponent a.e. Therefore $\mu\geq \mu_1$.

Let $G$ be a set s.t $\mu_1(G)=1$ and $(\mu-\mu_1)(G)=0$. Then $1>a=a\mu_1(G)=\mu(G)\geq \mu_1(G)=1$, a contradiction! Thus $\mu$ is ergodic and has positive entropy.
%	
%Next, we may assume that $x$ is a $\mu$-density point of $E_{n'}':=\{x\in G_\epsilon: \frac{\nu_{\xi^u(x)}(B^u(x,e^{-\delta n})\cap E_{n'})}{\nu_{\xi^u(x)}(B^u(x,e^{-\delta n}))}\geq e^{-\delta}\}$ so $\frac{\mu(B(x,e^{-\delta n})\cap E_{n'}')}{\mu(B(x,e^{-\delta n}))}\geq e^{-\delta}$, and denote the set of all such points by $G_\epsilon'$. Then for all $n$ large enough w.r.t $\epsilon$ and $\delta$, 
%\begin{equation}
%	\frac{1}{\mu(B(x,e^{-\delta n}))}\int_{B(x,e^{-\delta n})}g\circ f^n d\mu= e^{\pm\delta}\frac{1}{\mu(B(x,e^{-\delta n}))}\int\int \mathbb{1}_{B(x,e^{-\delta n})\cap E_{n''}'}g\circ f^n d\mu_{\xi^u(y)}d\mu(y) =e^{\pm \delta}
%\end{equation}	
\end{proof}

\noindent\textbf{Remark:} The weak physicality with full Basin of $\mu$ (Theorem \ref{PHYSICS}) also implies that $(M,f)$ may admit no physical measures aside for at most $\mu$.

\bibliographystyle{alpha}
\bibliography{Elphi}

\def\cprime{$'$} \def\cprime{$'$}
\begin{thebibliography}{KdlLPW01}

\bibitem[BORH]{NLE}
Snir Ben~Ovadia and Federico Rodriguez-Hertz.
\newblock Neutralized local entropy and dimension bounds for invariant
  measures.
\newblock Preprint: https://arxiv.org/abs/2302.10874.

\bibitem[Bow08]{B4}
Rufus Bowen.
\newblock {\em Equilibrium states and the ergodic theory of {A}nosov
  diffeomorphisms}, volume 470 of {\em Lecture Notes in Mathematics}.
\newblock Springer-Verlag, Berlin, revised edition, 2008.
\newblock With a preface by David Ruelle, Edited by Jean-Ren{\'e} Chazottes.

\bibitem[DKRH]{ExpMixBern}
Dmitry Dolgopyat, Adam Kanigowski, and Federico Rodriguez-Hertz.
\newblock Exponential mixing implies bernoulli.
\newblock Preprint: https://arxiv.org/abs/2106.03147.

\bibitem[HY95]{YoungCounterExample}
Hu~Yi Hu and Lai-Sang Young.
\newblock Nonexistence of {SBR} measures for some diffeomorphisms that are
  ``almost {A}nosov''.
\newblock {\em Ergodic Theory Dynam. Systems}, 15(1):67--76, 1995.

\bibitem[Kat23]{KatokBook23}
{\em A Vision for Dynamics in the 21st Century: The Legacy of Anatole Katok}.
\newblock Cambridge University Press, 2023.

\bibitem[KdlLPW01]{BrinHolderContFoliations}
Anatole Katok, Rafael de~la Llave, Yakov Pesin, and Howard Weiss, editors.
\newblock {\em Smooth ergodic theory and its applications}, volume~69 of {\em
  Proceedings of Symposia in Pure Mathematics}. American Mathematical Society,
  Providence, RI, 2001.

\bibitem[LY85]{LedrappierYoungI}
F.~Ledrappier and L.-S. Young.
\newblock The metric entropy of diffeomorphisms. {I}. {C}haracterization of
  measures satisfying {P}esin's entropy formula.
\newblock {\em Ann. of Math. (2)}, 122(3):509--539, 1985.

\bibitem[PSS]{PesinSenti}
Yakov Pesin, Samuel~Senti Senti, and Farruh Shahidi.
\newblock Area preserving surface diffeomorphisms with polynomial decay of
  correlations are ubiquitous.
\newblock https://arxiv.org/abs/2003.08503.

\bibitem[You02]{YoungSRBsurvey}
Lai-Sang Young.
\newblock What are {SRB} measures, and which dynamical systems have them?
\newblock {\em J. Statist. Phys.}, 108(5-6):733--754, 2002.
\newblock Dedicated to David Ruelle and Yasha Sinai on the occasion of their
  65th birthdays.

\end{thebibliography}

\end{document}